\documentclass[11pt]{article}
\usepackage{t1enc}
\usepackage{amsmath,amssymb,color}

\newtheorem{Thm}{Theorem}
\newtheorem{Def}{Definition}
\newtheorem{Prop}{Proposition}

\newtheorem{Lemma}{Lemma}
\newtheorem{Coro}{Corollary}
\newtheorem{Remark}{Remark}

\newenvironment{proof}[1][Proof]{\textbf{#1.} }{\hfill $\square$}

\newcommand{\cad}{c\`adl\`ag }
\newcommand{\what}{\widehat}
\newcommand{\wtil}{\widetilde}
\newcommand{\trace}{\mbox{Trace}}

\def \eps{\varepsilon}
\def \N{\mathbb{N}}
\def \R{\mathbb{R}}

\def \E{\mathbb{E}}
\def \F{\mathcal{F}}
\def \bF{\mathbb{F}}
\def \bM{\mathbb{M}}
\def \bL{\mathbb{L}}

\def \bH{\mathbb{H}}
\def \bD{\mathbb{D}}
\def \cE{\mathcal{E}}
\def \cP{\mathcal{P}}

\def \cU{\mathcal{U}}
\def \cT{\mathcal{T}}
\def \cD{\mathcal{D}}
\def \cM{\mathcal{M}}
\def \ds{\displaystyle}
\def \tP{\widetilde{\cP}}
\def \tOm{\widetilde{\Omega}}

\def \tpi{\widetilde{\pi}}

\def \al{\alpha}
\def \P{\mathbb{P}}
\def \1{\mathbf{1}}

\begin{document}

\title{BSDEs with monotone generator driven by Brownian and Poisson noises in a general filtration }
\author{T. Kruse \thanks{University of Duisburg-Essen, Thea-Leymann-Str. 9, 45127 Essen, Germany,
e-mail: {\tt thomas.kruse@uni-due.de}
}
, A. Popier \thanks{Universit\'e du
Maine, Laboratoire Manceau de Math\'ematiques, Avenue Olivier Messiaen, 72085 Le Mans, Cedex 9, France,
e-mail: {\tt alexandre.popier@univ-lemans.fr}
\hfill\break
We would like to thank two anonymous referees for helpful comments and suggestions. Thomas Kruse acknowledges the financial support from the French Banking Federation through the
Chaire "Markets in Transition".
}
}
\date{\today}

\maketitle

\begin{abstract}
We analyze multidimensional BSDEs in a filtration that supports a Brownian motion and a Poisson random measure. Under a monotonicity assumption on the driver, the paper extends several results from the literature. We establish existence and uniqueness of solutions in $L^p$ provided that the generator and the terminal condition satisfy appropriate integrability conditions. The analysis is first carried out under a deterministic time horizon, and then generalized to random time horizons given by a stopping time with respect to the underlying filtration. Moreover, we provide a comparison principle in dimension one.
\end{abstract}

\section*{Introduction}

The notion of nonlinear backward stochastic differential equations (BSDEs for short) was introduced by Pardoux and Peng \cite{pard:peng:90}. A solution of this equation, associated with a {\it terminal value} $\xi$ and a {\it generator or driver} $f(t,\omega,y,z)$, is a couple of stochastic processes $(Y_t,Z_t)_{t\leq T}$ such that
\begin{equation}\label{eq:standbsde}
Y_t=\xi+\int_t^Tf(s,Y_s,Z_s)ds-\int_t^TZ_sdW_s,
\end{equation}
a.s.\ for all $t\le T$, where $W$ is a Brownian motion and the processes $(Y_t,Z_t)_{t\leq T}$ are adapted to the natural filtration of $W$.

In their seminal work \cite{pard:peng:90}, Pardoux and Peng proved existence and uniqueness of a solution under suitable assumptions, mainly square integrability of $\xi$ and of the process $(f(t,\omega,0,0))_{t\leq T}$, on the one hand, and, the Lipschitz
property w.r.t. $(y,z)$ of the generator $f$, on the other hand. Since this first result, BSDEs have proved to be a powerful tool for formulating and solving a lot of mathematical problems arising for example in finance (see e.g. \cite{barr:elka:05, elka:peng:quen:97, roug:elka:00}), stochastic control and differential games (see e.g. \cite{hama:lepe:95,hama:lepe:peng:97}), or partial differential equations (see e.g. \cite{pard:99, pard:peng:92}).

\subsection*{Main results}

The aim of this paper is to establish existence and uniqueness of solutions to BSDE in a general filtration that supports a Brownian motion $W$ and an independent Poisson random measure $\pi$. We consider the following multi-dimensional BSDE:
\begin{equation} \label{eq:gene_BSDE}
Y_t = \xi + \int_t^T f(s,Y_s, Z_s,\psi_s) ds - \int_t^T\int_\cU \psi_s(u) \tpi(du,ds) -\int_t^TZ_sdW_s- \int_t^T dM_s.
\end{equation}
The solution is given by the usual triple $(Y,Z,\psi)$ and also an orthogonal local martingale $M$ which can not be reconstructed by the integrals w.r.t. to the Brownian and Poisson noise. We assume that the generator $f$ is monotonic (one-sided Lipschitz continuous) w.r.t. the $y$-variable and Lipschitz continuous w.r.t. to $z$ and $\psi$. Under the condition that the data $\xi$ and $f(t,0,0,0)$ are in $L^p$, $p > 1$, we provide existence and uniqueness results in $L^p$ spaces (the precise defintion will be given in Section \ref{sect:setting}). 

Further contributions are a comparison result in dimension one and existence and uniqueness when the terminal time is a non necessarily bounded stopping time.

\subsection*{Related literature}

There are already a lot of works which provide existence and uniqueness results under weaker assumptions than the ones of Pardoux and Peng \cite{pard:peng:90} or El Karoui et al \cite{elka:kapo:pard:97}. A huge part of the literature focuses on weakening the Lipschitz property of the coefficient $f$ w.r.t. the $y$-variable.
For example, Briand and Carmona \cite{bria:carm:00} and Pardoux \cite{pard:99} consider the case of a monotonic generator w.r.t. $y$ with different growth conditions. There have been relatively few papers which deal with the problem of existence and uniqueness of solutions in the case where the coefficients are not square integrable. El Karoui et al. \cite{elka:peng:quen:97} and Briand et al. \cite{bria:dely:hu:03} have proved existence and uniqueness of a solution for the standard BSDE \eqref{eq:standbsde} in the case where the data belong only to $L^p$ for some $p\geq 1$.

Another strand of research in the theory of BSDEs concerns the underlying filtration. In \cite{pard:peng:90} the filtration is generated by the Brownian motion $W$. Since the work of Tang and Li \cite{tang:li:94}, a lot of papers (see e.g. \cite{barl:buck:pard:97, bech:06, morl:10, pard:97, royer:06} or the books of Situ \cite{situ:05} or recently of Delong \cite{delo:13}) treat the case where the filtration is generated by the Brownian motion $W$ and a Poisson random measure $\pi$ independent of $W$. In most of these papers, the generator $f$ is supposed to be Lipschitz in $y$, even if the monotonic case is mentioned (see \cite{royer:06}) and all coefficients are square integrable. Yao \cite{yao:10} studies the $L^p$ case, $p>1$, and gives existence and uniqueness result in the case where the generator is monotone but with at most linear growth w.r.t. $y$. Li and Wei \cite{li:wei:14} give existence und uniqueness results for a fully coupled forward backward SDE under some monotonicity condition and $L^p$ coefficients, $p\geq 2$. Note that this monotonicity condition involves the coefficients of the forward diffusion and is not the same as the assumption imposed on the generator in this paper. An extension to BSDEs driven by a continuous local martingale $X$ and an integer-valued random measure $\pi$ has been studied by Xia \cite{xia:00}. Xia supposes that the filtration satisfies the representation property with respect to $X$ and $\pi$ and that the driver is Lipschitz continuous and square integrable.

For more general filtrations, the representation property of a local martingale is no more true (see Section III.4 in \cite{jaco:shir:03}) and an additional (orthogonal) martingale term has to be introduced in the definition of a solution. This approach was developed in the seminal work of El Karoui and Huang \cite{elka:huan:97} and by Carbone et al. \cite{carb:ferr:sant:07} for \cad martingales. The filtration $\bF$ is supposed to be complete, right continuous and quasi-left continuous. For a given square integrable martingale $X$ ($\langle X \rangle$ denotes the predictable projection of the quadratic variation), the BSDE \eqref{eq:standbsde} becomes
\begin{equation} \label{eq:bsde_gene_filt}
Y_t=\xi+\int_t^T f(s,Y_s,Z_s)d\langle X \rangle_s-\int_t^T Z_s dX_s - M_T + M_t.
\end{equation}
The solution is now the triple $(Y,Z,M)$ where $M$ is a square integrable martingale orthogonal to $X$. {\O}ksendal and Zhang \cite{okse:zhan:12} analyse BSDE of the form \eqref{eq:bsde_gene_filt} where $f$ does not depend on $z$, and apply to insider finance (see also Ceci et al. \cite{ceci:cret:russ:14}). Liang et al. \cite{lian:lyon:qian:11} also obtain results for a particular class of BSDE \eqref{eq:bsde_gene_filt} on an arbitrary filtered probability space. In these papers, existence and uniqueness of the solution of \eqref{eq:bsde_gene_filt} is proved for a Lipschitz continuous function $f$ and under square integrability condition (in \cite{okse:zhan:12} the monotone case is treated but $f$ does not depend on $z$). The Hilbertian structure of $L^2(\Omega,\F_T,\P)$ is used in Cohen and Elliott \cite{cohe:elli:12} (see also \cite{klim:14}). If $L^2(\Omega,\F_T,\P)$ is a separable Hilbert space, then an orthogonal basis of martingales can be introduced instead of $X$ and there is no orthogonal additional term $M$ in \eqref{eq:bsde_gene_filt}. $Z$ becomes a sequence of predictable processes. The special case of a L\'evy noise is treated before by Nualart and Schoutens \cite{nual:scho:01}: the orthogonal basis of martingales is explicitely given by the Teugels martingales.  

In very recent papers, Klimsiak has developed the results concerning BSDEs in this general framework in two directions. First for reflected BSDE (\cite{klim:13b,klim:14}), and secondly for parabolic equations (\cite{klim:rozk:13,klim:13c}) with measure data. 

\subsection*{Main contributions}

Let us outline the main contributions of our paper compared to the existing literature.  

First of all our paper generalizes many results from the works \cite{barl:buck:pard:97, bech:06, morl:10, pard:97, royer:06, tang:li:94} dealing with a filtration generated by the Brownian motion and the Poisson random measure since we allow for a more general filtration. 

Moreover we provide existence and uniqueness of solutions in $L^p$-spaces, $p>1$. In the case where the generator depends on the stochastic integrand w.r.t.\ a Poisson random measure, the case when $p < 2$ has to be handled carefully and can not be treated as in \cite{bria:dely:hu:03}. Indeed in this case Burkholder-Davis-Gundy inequality with $p/2<1$ does not apply and the $L^{p/2}$-norm of the predictable projection cannot be controlled by the $L^{p/2}$-norm of the quadratic variation (see Inequality \eqref{eq:pred_quad_var} and \cite{leng:lepi:prat:80}). Yao \cite{yao:10} obtains similar results but for a generator with at most linear growth w.r.t. $y$ (and for a filtration generated by $W$ and $\pi$). Klimsiak \cite{klim:14} considers $L^p$ solutions of BSDE, with $p\neq 2$, in a general filtration but where the driver only depends on $y$.

Compared to \cite{carb:ferr:sant:07} or \cite{xia:00}, our assumptions are in some sense more restrictive as we assume that the continuous part of the given martingale $X$ of BSDE \eqref{eq:bsde_gene_filt} is a Brownian motion $W$ and the random measure associated to the jumps of $X$ is a Poisson random measure $\pi$. However we weaken the assumptions on the driver and on the terminal condition: the generator is only supposed to be monotone and the terminal condition is allowed to be only $L^p$-integrable. To the best of our knowledge, there is no existence and uniquenesss result for multi-dimensional BSDE with $L^p$ coefficients in a general filtration. The generalization of our results for BSDE of the form \eqref{eq:bsde_gene_filt} requires some sophisticated integrability conditions to take account of the predictable projection $\langle X \rangle$ of the quadratic variation of $X$. Therefore it is left for future research.

Moreover we provide a comparison principle and existence and uniqueness in the case with random terminal time for BSDE of type \eqref{eq:gene_BSDE}. The proof of the comparison principle generalizes the arguments of \cite{quen:sule:13} to the situation where the filtration is not only generated by Brownian and Poisson noise.

Finally our setting is important for the control problem we study in the paper \cite{krus:popi:15} (see also \cite{grae:hors:qiu:13}). The control problem arises in mathematical finance and models the optimal liquidation of a financial position in an illiquid market. In \cite{grae:hors:qiu:13} the authors consider the case when the filtration is generated by a Brownian motion and a independent Poisson measure. In \cite{krus:popi:15} we do not impose any condition on the filtration generated by the market (except right-continuity, completeness and quasi-left continuity) and the Poisson random measure represents the limit orders of the trading strategy. The optimal strategy is determined by a solution of a BSDE of the form \eqref{eq:gene_BSDE} where the generator depends on the Poisson random measure.

\subsection*{Decomposition of the paper}

The paper is decomposed as follows. In the first section, we give the mathematical setting and the main results of this paper. In the second part, we consider square integrable coefficients and we prove existence and uniqueness of the solution. To prove it we mainly follow the scheme of \cite{pard:99} with suitable modifications. In the next part, we extend the result to $L^p$ coeffcients for any $p> 1$. For $p> 2$, the existence is derived from the existence in the $L^2$ case with the right a priori estimate. For $1< p < 2$, an extra computation has to be made since the function $x \mapsto |x|^p$ is not smooth in this case. We have to extend Lemma 2.2 in \cite{bria:dely:hu:03} or Proposition 2.1 in \cite{klim:13b} to our framework. In the last two sections, we add two extensions: the comparison result in dimension one, and existence and uniqueness when the terminal time is a stopping time. Note that the comparison principle requires an extra condition when the generator depends on the jump part (see the counterexample in \cite{barl:buck:pard:97}). But instead of using Girsanov's theorem to obtain the comparison between two solutions, we generalize the argument of \cite{quen:sule:13}, which is less restrictive. This point will be crucial in \cite{krus:popi:15}.

\section{Settings and main results} \label{sect:setting}

Let us consider a filtered probability space $(\Omega,\F,\P,\bF = (\F_t)_{t\geq 0})$. The filtration is assumed to be complete, right continuous and quasi-left continuous, which means that for every sequence $(\tau_n)$ of $\bF$ stopping times such that $\tau_n \nearrow \tilde \tau$ for some stopping time $\tilde \tau$ we have $\bigvee_{n\in \N}\F_{\tau_n}=\F_{\tilde \tau}$.. 
Without loss of generality we suppose that all semimartingales have right continuous paths with left limits. We assume that $(\Omega,\F,\P,\bF = (\F_t)_{t\geq 0})$ supports a $k$-dimensional Brownian motion $W$ and a Poisson random measure $\pi$ with intensity $\mu(du)dt$ on the space $\cU \subset \R^m \setminus \{0\} $. The measure $\mu$ is $\sigma$-finite on $\cU$ such that
$$\int_\cU (1\wedge |u|^2) \mu(du) <+\infty.$$
The compensated Poisson random measure $\tpi(du,dt) = \pi(du,dt) - \mu(du) dt$ is a martingale w.r.t.\ the filtration $\bF$.

In this paper for a given $T\geq 0$, we denote:
\begin{itemize}
\item $\cP$: the predictable $\sigma$-field on $\Omega \times [0,T]$ and
$$\tP=\cP \otimes \mathcal{B}(\cU)$$
where $\mathcal{B}(\cU)$ is the Borelian $\sigma$-field on $\cU$.
\item On $\tOm = \Omega \times [0,T] \times \cU$, a function that is $\tP$-measurable, is called predictable. $G_{loc}(\pi)$ is the set of $\tP$-mesurables functions $\psi$ on $\tOm$ such that for any $t \geq 0$ a.s.
$$ \int_0^t \int_\cU (|\psi_s(u)|^2\wedge |\psi_s(u)|) \mu(du) < +\infty.$$
\item $\cD$ (resp. $\cD(0,T)$): the set of all predictable processes on $\R_+$ (resp. on $[0,T]$). $L^2_{loc}(W)$ is the subspace of $\cD$ such that for any $t\geq 0$ a.s.
$$\int_0^t |Z_s|^2 ds < +\infty.$$
\item $ \cM_{loc}$: the set of c\`adl\`ag local martingales orthogonal to $W$ and $\tpi$. If $M \in \cM_{loc}$ then
$$[ M, W^i]_t =0, 1\leq i \leq k \qquad [M ,\tpi(A,.)]_t = 0$$
for all $A\in \mathcal B(\mathcal U)$.
In other words, $\E (\Delta M * \pi | \tP) = 0$, where the product $*$ denotes the integral process (see II.1.5 in \cite{jaco:shir:03}). Roughly speaking, the jumps of $M$ and $\pi$ are independent. 
\item $\cM$ is the subspace of $ \cM_{loc}$ of martingales.
\end{itemize}
We refer to \cite{jaco:shir:03} (see also \cite{bech:06}) for details on random measures and stochastic integrals. As explained above, the filtration $\bF$ supports the Brownian motion $W$ and the Poisson random measure $\pi$. 
\begin{Lemma}[Lemma III.4.24 in \cite{jaco:shir:03}] \label{lem:decomp_loc_mart}
Every local martingale has a decomposition
$$\int_0^. Z_s dW_s + \int_0^. \int_\cU \psi_s(u) \tpi(du,ds) + M$$
where $M \in \cM_{loc}$, $Z \in L^2_{loc}(W)$, $\psi \in G_{loc}(\mu)$.
\end{Lemma}
Now to define the solution of our BSDE, let us introduce the following spaces for $p\geq 1$.
\begin{itemize}
\item $\bD^p(0,T)$ is the space of all adapted \cad processes $X$ such that
$$\E \left(  \sup_{t\in [0,T]} |X_t|^p \right) < +\infty.$$
For simplicity, $X_* = \sup_{t\in [0,T]} |X_t|$.
\item $\bH^p(0,T)$ is the subspace of all processes $X\in \cD(0,T)$ such that
$$\E \left[ \left( \int_0^T |X_t|^2 dt\right)^{p/2} \right] < +\infty.$$
\item $\bM^p(0,T)$ is the subspace of $\cM$ of all martingales such that
$$\E \left[ \left( [ M ]_T \right)^{p/2}\right] < +\infty.$$
\item $\bL^p_\pi(0,T) = \bL^p_{\pi}(\Omega\times (0,T)\times \cU)$: the set of processes $\psi \in G_{loc}(\mu)$ such that
$$\E \left[ \left(  \int_0^T \int_{\cU} |\psi_s(u)|^2 \mu(du) ds \right)^{p/2} \right] < +\infty .$$
\item $\bL^p_\mu=\bL^p(\cU,\mu;\R^d)$: the set of measurable functions $\psi : \cU \to \R^d$ such that
$$\| \psi \|^p_{\bL^p_\mu} = \int_{\cU} |\psi(u)|^p \mu(du)  < +\infty .$$
\item $\cT$ : the set of all finite stopping times and $\cT_T$ the set of all stopping times with values in $[0,T]$.
\item $\cE^p(0,T) = \bD^p(0,T) \times \bH^p(0,T) \times \bL^p_\pi(0,T) \times \bM^p(0,T)$.
\end{itemize}

We consider the following BSDE \eqref{eq:gene_BSDE}:
\begin{equation*} 
Y_t = \xi + \int_t^T f(s,Y_s, Z_s,\psi_s) ds - \int_t^T\int_\cU \psi_s(u) \tpi(du,ds) -\int_t^TZ_sdW_s- \int_t^T dM_s.
\end{equation*}
Here, the random variable $\xi$ is $\F_T$-measurable with values in $\R^d$ ($d\geq 1$) and the generator $f : \Omega \times [0,T] \times \R^d \times \R^{d\times k} \times \bL^2_\mu \to \R^d$ is a random function, measurable with respect to $Prog \times \mathcal{B}(\R^d)\times \mathcal{B}(\R^{d \times k}) \times \mathcal{B}(\bL^2_\mu)$ where $Prog$ denotes the sigma-field of progressive subsets of $\Omega \times [0,T]$.

The unknowns are $(Y,Z,\psi,M)$ such that
\begin{itemize}
\item $Y$ is progressively measurable and \cad with values in $\R^d$;
\item $Z \in L^2_{loc}(W)$, with values in $\R^{d\times k}$;
\item $\psi \in G_{loc}(\mu)$ with values in $\R^d$;
\item $M \in \cM_{loc}$ with values in $\R^d$.
\end{itemize}
On $\R^d$, $|.|$ denotes the Euclidean norm and $\R^{d\times k}$ is identified with the space of real matrices with $d$ rows and $k$ columns. If $z \in  \R^{d\times k}$, we have $|z|^2 = \trace(zz^*)$. If $M$ is a $\R^d$-valued martingale in $\cM$, the bracket process $[ M ]_t$ is
$$[ M ]_t = \sum_{i=1}^d [ M^i ]_t,$$
where $M^i$ is the $i$-th component of the vector $M$.

Throughout the paper, the following assumptions on the generator $f$ are denoted by ${\bf (H_{ex})}$.
\begin{enumerate}
\item[(H1)] For every $t \in [0,T]$, $z\in \R^{d\times k}$ and every $\psi \in \bL^2_\mu$ the mapping $y \in \R^d \mapsto f(t,y,z,\psi)$ is continuous. Moreover there exists a constant
$\alpha$ such that
$$\langle f(t,y,z,\psi)-f(t,y',z,\psi),y-y'\rangle \leq \alpha |y-y'|^2.$$
\item[(H2)] For every $r  > 0$ the mapping $(\omega,t) \mapsto \sup_{|y|\leq r} |f(t,y,0,0)-f(t,0,0,0)|$ belongs to $L^1(\Omega \times [0,T], \P \otimes m)$.
\item[(H3)] $f$ is Lipschitz continuous w.r.t.\ $z$ and $\psi$: there exists a constant $K$ such that for any $t$ and $y$, for any $z,z' \in \R$ and $\psi,\psi'$ in $\bL^2(\mu)$
$$|f(t,y,z,\psi)-f(t,y,z',\psi')| \leq K (|z-z'|+\|\psi- \psi'\|_{\bL^2_\mu}).$$
\end{enumerate}
We can suppose w.l.o.g.\ that $\alpha=0$. Indeed if $(Y,Z,\psi,M)$ is a solution of \eqref{eq:gene_BSDE} then $(\bar Y,\bar Z, \bar \psi, \bar M)$ with
$$\bar Y_t = e^{\alpha t} Y_t, \quad \bar Z_t = e^{\alpha t} Z_t, \quad \bar \psi_t  = e^{\alpha t} \psi_t, \quad d \bar M_t = e^{\al t} dM_t $$
satisfies an analogous BSDE with terminal condition $\bar \xi = e^{\alpha T}\xi$ and generator
$$\bar f(t,y,z,\psi) = e^{\al t} f(t,e^{-\al t} y, e^{-\al t}z, e^{-\al t} \psi) - \al y.$$
$\bar f$ satisfies assumptions ${\bf (H_{ex})}$ with $\al = 0$. Hence in the rest of this paper, we will suppose that $\al = 0$.

Our main results can be summarized as follows. Under Assumptions ${\bf (H_{ex})}$ and if for some $p> 1$
\begin{equation*}
 \E \left(|\xi|^p +  \int_0^T |f(t,0,0,0)|^p dt \right) < +\infty,
 \end{equation*}
there exists a unique solution $(Y,Z,\psi,M)$ in $\cE^p(0,T)$ to the BSDE \eqref{eq:gene_BSDE}. The comparison principle holds for this BSDE. Moreover with a suitable conditions (see (H5') and (H6)) the terminal time $T$ can be replaced by a stopping time $\tau$.

\section{$L^2$ solutions}

Let us begin with the following definition:
\begin{Def}[$L^2$-solution] \label{def:L2_sol_BSDE}
We say that $(Y,Z,\psi,M)$ is a $L^2$-solution of the BSDE $(\xi,f)$ on $[0,T]$ if
\begin{itemize}
\item $(Y,Z,\psi,M) \in \cE^2(0,T)$;
\item and Equation \eqref{eq:gene_BSDE} is satisfied $\P \otimes dt$-a.s.
\end{itemize}
\end{Def}
In the rest of this section, a solution $(Y,Z,\psi,M)$ will be supposed to be in $\cE^2(0,T)$. Now we want to prove existence of the solution of the BSDE with data $(\xi,f)$. For this purpose, we will add the integrability conditions:
\begin{enumerate}
\item[(H4)] $\ds \E (|\xi|^2) + \E \int_0^T |f(t,0,0,0)|^2 dt  < +\infty$.
\end{enumerate}
Some a priori estimates are needed. Note that the following results are modifications of the results obtained in \cite{pard:97}.
\begin{Lemma} \label{lem:estim_mart_part}
Let $(Y,Z,\psi,M) \in \cE^2(0,T)$ be a solution of BSDE \eqref{eq:gene_BSDE}. Then
\begin{eqnarray*}
\E \left( \int_0^T |Z_s|^2 ds +  \int_0^T \int_\cU |\psi_s(u)|^2 \mu(du) ds + [ M ]_T \right) & \leq & C  \E \left( \sup_{t\in[0,T]} |Y_t|^2 + \int_0^T |f(t,0,0,0)|^2 dt  \right)
\end{eqnarray*}
for some constant $C$ depending only on $K^2$ and $T$.
\end{Lemma}
\begin{proof}
Let $\tau \in \cT_T$ and by It\^o's formula on $|Y_t|^2$:
\begin{eqnarray} \nonumber
&& |Y_{\tau \wedge t}|^2 + \int_{\tau \wedge t}^{\tau} |Z_s|^2 ds + \int_{\tau \wedge t}^{\tau} \int_\cU |\psi_s(u)|^2 \mu(du) ds + [ M]_{\tau} - [ M]_{\tau \wedge t} \\ \nonumber
&& \quad =  |Y_{\tau}|^2 + 2 \int_{\tau \wedge t}^{\tau} Y_s f(s,Y_s,Z_s,\psi_s) ds -2 \int_{\tau \wedge t}^{\tau} Y_{s^-} Z_s dW_s \\ \label{eq:Ito_formula}
&& \qquad  -2 \int_{\tau \wedge t}^{\tau} Y_{s^-} dM_s  -\int_{\tau \wedge t}^{\tau} \int_\cU \left( |Y_{s^-} + \psi_s(u) |^2 - |Y_{s^-}|^2 \right) \tpi(du,ds).
\end{eqnarray}
But from (H1) and (H3):
$$yf(t,y,z,\psi)  \leq |y| \left( K |z| + K \| \psi\|_{\bL^2_\mu} + |f(t,0,0,0)|\right).$$
Hence with $t=0$ and Young's inequality:
\begin{eqnarray*}
&&\frac{1}{2}  \int_0^\tau  |Z_s|^2 ds + \frac{1}{2}  \int_0^\tau \int_\cU |\psi_s(u)|^2 \mu(du) ds +  [M]_\tau \\
&&\quad \leq ((4K^2+1)T+1) \sup_{t\in [0,T]}  |Y_t|^2 + \int_0^T  |f(s,0,0,0)|^2 ds -2 \int_{\tau \wedge t}^{\tau} Y_{s^-} Z_s dW_s \\
&& \qquad -2 \int_0^\tau Y_{s^-} dM_s  -\int_0^\tau \int_\cU \left( |Y_{s^-} + \psi_s(u) |^2 - |Y_{s^-}|^2 \right) \tpi(du,ds).
\end{eqnarray*}
Moreover from the assumptions on $Y$, $Z$, $\psi$ and $M$, the stochastic integral terms w.r.t.\ $W$, $M$ and $\tpi$ are martingales. Now take $\tau=T$ and we can take the expectation on both parts:
\begin{eqnarray*}
&&\E \left( \frac{1}{2}  \int_0^T  |Z_s|^2 ds + \frac{1}{2} \int_0^{T} \int_\cU |\psi_s(u)|^2 \mu(du) ds +  [ M]_{T} \right) \\
&& \qquad  \leq (4K^2+2)T \E \left( \sup_{t\in [0,T]}  |Y_t|^2\right) + \E \int_0^T | f(s,0,0,0)|^2 ds
\end{eqnarray*}
which achieves the proof.
\end{proof}
\begin{Lemma} \label{lem:estim_Y_part}
Let $(Y,Z,\psi,M)$ be a solution of BSDE $(\xi,f)$ with the same conditions as in Lemma \ref{lem:estim_mart_part}.
Then
\begin{eqnarray*}
\E\left(  \sup_{t\in [0,T]} |Y_t|^2 \right) & \leq & C  \E \left( |\xi|^2 + \int_0^T |f(t,0,0,0)|^2 dt  \right)
\end{eqnarray*}
for some constant $C$ depending only on $K$ and $T$.
\end{Lemma}
\begin{proof}
For $0\leq t \leq T$, let us apply It\^o formula \eqref{eq:Ito_formula} to $e^{\beta s} |Y_s|^2$ where $\beta$ will be chosen later. We have:
\begin{eqnarray*}
e^{\beta t}|Y_{t}|^2 &= & e^{\beta T}|Y_{T}|^2+2 \int_{t}^{T}  e^{\beta s}Y_s f(s,Y_s,Z_s,\psi_s) ds  \\
& - &\int_{t }^{T}  \beta e^{\beta s}|Y_s|^2 ds - \int_{t}^{T}  \int_\cU e^{\beta s} | \psi_s(u) |^2 \mu(du) ds\\
& - &  \int_{t}^{T} e^{\beta s} d[ M ]_s -\int_{t}^{T}  e^{\beta s} |Z_s|^2 ds \\
&- &  2  \int_{t}^{T} e^{\beta s}Y_{s^-}  dM_s - 2  \int_{t}^{T} e^{\beta s}Y_{s^-}  Z_s dW_s   \\
& - &   \int_{t}^{T} e^{\beta s} \int_\cU  \left( |Y_{s^-} + \psi_s(u) |^2 - |Y_{s^-}|^2 \right) \tpi(du,ds).
\end{eqnarray*}
From the assumptions on $f$, we have for any $\eps > 0$
\begin{eqnarray} \nonumber
Y_s f(s,Y_s,Z_s,\psi_s) &= & Y_s (f(s,Y_s,Z_s,\psi_s)- f(s,0,Z_s,\psi_s)) + Y_s (f(s,0,Z_s,\psi_s)- f(s,0,0,\psi_s)) \\ \nonumber
&+& Y_s (f(s,0,0,\psi_s)- f(s,0,0,0))+  Y_sf(s,0,0,0)\\ \nonumber
& \leq & |Y_s| |f(s,0,0,0)| + K |Y_s| \|\psi_s\|_{\bL^2_\mu}  + K|Y_s| |Z_s| \\ \label{eq:young_growth_f}
& \leq & \frac{(1+2K^2)}{2\eps} |Y_s|^2 + \frac{\eps}{2} ( |f(s,0,0,0)|^2 + \|\psi_s\|_{\bL^2_\mu}^2 + |Z_s|^2).
\end{eqnarray}
We take $\eps = 1/2$ and we obtain:
\begin{eqnarray*}
&& e^{\beta t}|Y_{t}|^2 + \frac{1}{2} \int_{t}^{T} e^{\beta s} |Z_s|^2 ds + \frac{1}{2}  \int_{t}^{T}  e^{\beta s} \| \psi_s\|^2_{\bL^2_\mu}  ds + \frac{1}{2} \int_{t}^{T} e^{\beta s} d[ M]_s\\
&& \quad \leq e^{\beta T}|Y_{T}|^2 +\frac{1}{2}  \int_{t}^{T }  e^{\beta s} |f(s,0,0,0)|^2 ds   \\
&& \qquad - \int_{t}^{T}  \left( \beta - 2(1+2K^2)\right) e^{\beta s}|Y_s|^2 ds + \Gamma_{t,T}
\end{eqnarray*}
where $\Gamma_{t,.}$ is a local martingale starting at zero at time $t$. Fix $\beta = 2(1+2K^2)$ and we have:
\begin{eqnarray}\nonumber
&& e^{\beta t}|Y_{t}|^2 + \frac{1}{2}  \int_{t}^{T }  e^{\beta s} \| \psi_s\|^2_{\bL^2_\mu}  ds + \frac{1}{2} \int_{t}^{T} e^{\beta s} d[M]_s\\ \label{eq:L2_estimate1}
&& \quad \leq e^{\beta T}|Y_{T}|^2 +\frac{1}{2}  \int_{t}^{T}  e^{\beta s} |f(s,0,0,0)|^2 ds + \Gamma_{t,T} .
\end{eqnarray}
Since all local martingales are true martingales, we deduce that
\begin{eqnarray} \nonumber
&&\sup_{t\in [0,T]}  \E |Y_t|^2 + \E \int_0^T  | Z_s|^2 ds  + \E \int_0^T  \| \psi_s\|^2_{\bL^2_\mu}  ds + \E [M]_T \\ \label{eq:L2_estimate_bis}
&& \qquad \leq C  \E \left( |\xi|^2 + \int_0^T |f(t,0,0,0)|^2 dt  \right).
\end{eqnarray}
Now with $\beta = 2(1+2K^2)$ we write the It\^o formula in a different way:
\begin{eqnarray*}
e^{\beta t}|Y_{t}|^2 &= & e^{\beta T}|Y_{T}|^2+2 \int_{t}^{T}  e^{\beta s}Y_s f(s,Y_s,Z_s,\psi_s) ds  \\
& - &\int_{t }^{T}  \beta e^{\beta s}|Y_s|^2 ds - \int_{t}^{T}  \int_\cU e^{\beta s} | \psi_s(u) |^2 \pi(du, ds)\\
& - &  \int_{t}^{T} e^{\beta s} d[ M ]_s -\int_{t}^{T}  e^{\beta s} |Z_s|^2 ds \\
&- &  2  \int_{t}^{T} e^{\beta s}Y_{s^-}  dM_s - 2  \int_{t}^{T} e^{\beta s}Y_{s^-}  Z_s dW_s   \\
& - &  2 \int_{t}^{T} e^{\beta s} \int_\cU  Y_{s^-} \psi_s(u) \tpi(du,ds)
\end{eqnarray*}
which gives with \eqref{eq:young_growth_f}:
\begin{eqnarray}\nonumber
&&e^{\beta t}|Y_{t}|^2 \leq e^{\beta T}|Y_{T}|^2 +\frac{1}{2}  \int_{t}^{T}  e^{\beta s} |f(s,0,0,0)|^2 ds + \frac{1}{4}  \int_{t}^{T }  e^{\beta s} \| \psi_s\|^2_{\bL^2_\mu}  ds \\ \label{eq:L2_estimate_1_bis}
&& \quad -  2  \int_{t}^{T} e^{\beta s}Y_{s^-}  dM_s - 2  \int_{t}^{T} e^{\beta s}Y_{s^-}  Z_s dW_s - 2 \int_{t}^{T} e^{\beta s} \int_\cU  Y_{s^-} \psi_s(u) \tpi(du,ds)
\end{eqnarray}
Next we apply the Burkholder-Davis-Gundy inequality to obtain
\begin{eqnarray} \nonumber
&& \E \sup_{t\in [0,T]} \left|  \int_{0}^{t} e^{\beta s} Y_{s^-} dM_s \right|  \leq c \E \left(   \int_{0}^{T} e^{2\beta s} |Y_{s^-}|^{2} d[M ]_s \right)^{1/2} \\ \label{eq:L2_estimate_2}
&& \qquad \qquad \leq \frac{1}{8} \E \left( \sup_{t \in [0,T]}  e^{\beta t}| Y_{t}|^{2}  \right) + 2 c^2 \E  \int_{0}^{T} e^{\beta s}d[ M ]_s.
\end{eqnarray}
By the same arguments we have
\begin{eqnarray} \nonumber
&& \E \sup_{t\in [0,T]} \left|  \int_{0}^{t} e^{\beta s}Y_{s^-} Z_s dW_s \right|  \leq c \E \left(   \int_{0}^{T} e^{2\beta s} |Y_{s^-}|^{2}  | Z_s|^2 ds \right)^{1/2} \\  \label{eq:L2_estimate_3}
&& \qquad \qquad \leq \frac{1}{8} \E \left( \sup_{t \in [0,T]}  e^{\beta t}| Y_{t}|^{2}  \right) + 2c^2 \E  \int_{0}^{T} e^{\beta s} |Z_s|^2 ds .
\end{eqnarray}
Finally the same result holds for the martingale
$$\int_{0}^{.} \int_\cU \left( Y_{s^-} \psi_s(u) \right) \tpi(du,ds),$$
with
\begin{eqnarray} \nonumber
&& \E \sup_{t\in [0,T]} \left| \int_{0}^{t} e^{\beta s} \int_\cU \left( Y_{s^-} \psi_s(u) \right) \tpi(du,ds) \right|  \leq c \E \left(   \int_{0}^{T} e^{2\beta s}|Y_{s^-}|^{2}\int_\cU | \psi_s(u)|^2 \pi(du,ds) \right)^{1/2} \\  \label{eq:L2_estimate_4}
&& \qquad \qquad \leq \frac{1}{8} \E \left( \sup_{t \in [0,T]} e^{\beta t} |Y_{t}|^{2}  \right) + 2c^2 \E  \int_{0}^{T} e^{\beta s}\| \psi_s\|_{\bL^2_\mu}^2 ds.
\end{eqnarray}
Now coming back to \eqref{eq:L2_estimate_1_bis}, and using estimates \eqref{eq:L2_estimate_2}, \eqref{eq:L2_estimate_3}, \eqref{eq:L2_estimate_4} and \eqref{eq:L2_estimate_bis}, we deduce that there exists $C$ depending on $K$ and $T$ such that
\begin{eqnarray*}
\E\left(  \sup_{t\in [0,T]} |Y_t|^2 \right) & \leq & C  \E \left( |\xi|^2 + \int_0^T |f(t,0,0,0)|^2 dt  \right).
\end{eqnarray*}
This achieves the proof. 
\end{proof}

The next result is an extension of the Proposition 2.1 in \cite{bria:carm:00}. For convenience let us give the result and the proof.
\begin{Lemma} \label{lem:Linfinity_estim}
Let $(Y,Z,\psi,M) \in \cE^2(0,T)$ be a solution of BSDE \eqref{eq:gene_BSDE} with bounded terminal condition $\xi$ and generator $|f(t,0,0,0)|$: there exists a constant $\kappa$ such that a.s.
\begin{equation} \label{eq:bounded_assump}
\sup_{t\in [0,T]} |f(t,0,0,0)| + |\xi| \leq \kappa.
\end{equation}
Then $Y$ is also almost surely bounded: there exists a constant $\beta = 2(1+2K^2)$ such that almost surely and for any $t \in [0,T]$
\begin{eqnarray*}
|Y_t|^2 \leq \kappa^2 e^{\beta (T-t)} \left( 1 + \frac{1}{2\beta} \right).
\end{eqnarray*}
\end{Lemma}
\begin{proof}
We use Inequality \eqref{eq:L2_estimate1} and since 
the involved local martingale $\Gamma$ is a martingale, taking the conditional expectation w.r.t.\ $\F_t$ leads to: a.s. for every $t \in [0,T]$
\begin{eqnarray*}
&& |Y_{t}|^2 \leq \E\left[ e^{\beta(T-t)}|\xi|^2 +\frac{1}{2} \int_{t}^{T}  e^{\beta(s-t)} |f(s,0,0,0)|^2 ds  \bigg| \F_t \right].
\end{eqnarray*}
Hence $Y \in \cD^{\infty}(0,T)$.
\end{proof}

Now we prove a stability result. 
\begin{Lemma} \label{lem:L2_stability}
Let now $(\xi,f)$ and $(\xi',f')$ be two sets of data each satisfying the above assumptions ${\bf (H_{ex})}$ and (H4). Let $(Y,Z,\psi,M)$ (resp. $(Y',Z',\psi',M')$) denote a $L^2$-solution of the BSDE \eqref{eq:gene_BSDE} with data $(\xi,f)$ (resp. $(\xi',f')$). Define
$$(\what Y,\what Z, \what \psi, \what M, \what \xi, \what f) = (Y-Y',Z-Z',\psi-\psi', M-M', \xi - \xi', f-f').$$
Then there exists a constant $C$ depending on $K^2$ and $T$, such that
\begin{eqnarray*}
&& \E\left(  \sup_{t\in [0,T]} |\what Y_t|^2 + \int_0^T |\what Z_s|^2 ds + \int_0^T \int_\cU |\what \psi_s(u)|^2 \mu(du) ds + [ \what M ]_T \right) \\
&& \qquad \leq  C  \E \left( |\what \xi|^2 + \int_0^T |\what f(t,Y'_t,Z'_t,\psi'_t) |^2 dt  \right).
\end{eqnarray*}
\end{Lemma}
As a consequence of this lemma, we obtain uniqueness of the solution $(Y,Z,\psi,M)$ for the BSDE \eqref{eq:gene_BSDE} in the set $\cE^2(0,T)$ (see also Corollary \ref{coro:uniq_sol} in dimension $d=1$).

\noindent \begin{proof}
Let $\tau \in \cT_T$ and by It\^o's formula on $|\what Y_t|^2$:
\begin{eqnarray*} \nonumber
&& |\what Y_{\tau \wedge t}|^2  + \int_{\tau \wedge t}^{\tau} |\what Z_s|^2 ds + \int_{\tau \wedge t}^{\tau} \int_\cU |\what \psi_s(u)|^2 \mu(du) ds +  [ \what M]_{\tau} -[ \what M]_{\tau \wedge t} \\ \nonumber
&& =  |\what Y_{\tau}|^2 + 2 \int_{\tau \wedge t}^{\tau} \what Y_s (f(s,Y_s,Z_s,\psi_s)-f'(s,Y'_s,Z'_s,\psi'_s)) ds - 2 \int_{\tau \wedge t}^{\tau} \what Y_{s^-} \what Z_s dW_s\\
&& \qquad -2 \int_{\tau \wedge t}^{\tau} \what Y_{s^-} d\what M_s  -\int_{\tau \wedge t}^{\tau} \int_\cU \left( |\what Y_{s^-} +\what  \psi_s(u) |^2 - |\what Y_{s^-}|^2 \right) \tpi(du,ds).
\end{eqnarray*}
From the monotonicity assumption on the generator and Young's inequality, we have:
\begin{eqnarray} \nonumber
&& |\what Y_{\tau \wedge t}|^2 +\frac{1}{2} \int_{\tau \wedge t}^{\tau} |\what Z_s|^2 ds + \frac{1}{2}\int_{\tau \wedge t}^{\tau} \int_\cU |\what \psi_s(u)|^2 \mu(du) ds + [ \what M]_{\tau} -[ \what M]_{\tau \wedge t} \\ \nonumber
&& \leq  |\what Y_{\tau}|^2 + (4K^2+1) \int_{\tau \wedge t}^{\tau} |\what Y_s|^2 ds + \int_{\tau \wedge t}^{\tau}  |\what f(s,Y'_s,Z'_s,\psi'_s)|^2 ds - 2 \int_{\tau \wedge t}^{\tau} \what Y_{s^-} \what Z_s dW_s\\ \label{eq:Ito_formula_2}
&& \qquad -2 \int_{\tau \wedge t}^{\tau} \what Y_{s^-} d\what M_s  -\int_{\tau \wedge t}^{\tau} \int_\cU \left( |\what Y_{s^-} +\what  \psi_s(u) |^2 - |\what Y_{s^-}|^2 \right) \tpi(du,ds).
\end{eqnarray}
With $\tau=T$ and Gronwall's lemma, we have for any $t\in [0,T]$
$$\E  |\what Y_t|^2 \leq C  \E \left( |\what \xi|^2 + \int_0^T |\what f(s,Y'_s,Z'_s,\psi'_s) |^2 ds  \right).$$
Then using \eqref{eq:Ito_formula_2} with $t=0$ and $\tau=T$ and the previous inequality we obtain
$$\E\left( \int_0^T |\what Z_s|^2ds +  \int_0^{T} \int_\cU |\what \psi_s(u)|^2 \mu(du) ds + [\what M]_{T} \right) \leq  C  \E \left( |\what \xi|^2 + \int_0^T |\what f(s,Y'_s,Z'_s,\psi'_s) |^2 ds  \right).$$
Finally take the conditional expectation w.r.t.\ $\F_t$ in \eqref{eq:Ito_formula_2}, the supremum over $t \in [0,T]$ on both sides and applying Doob's inequality to the supremum of the $(\F_{\tau \wedge t}, \ t\in[0,T])$ martingale on the right-hand side, we have:
$$\E \left( \sup_{t\in [0,T]} |\what Y_t|^2 \right) \leq C  \E \left( |\what \xi|^2 + \int_0^T |\what f(s,Y'_s,Z'_s,\psi'_s) |^2 ds  \right).$$
This completes the proof.
\end{proof}

Let us modify a little the growth assumption (H2):
\begin{enumerate}
\item[(H2')] For every $(t,y) \in [0,T]\times \R^d$, $ |f(t,y,0,0)| \leq |f(t,0,0,0)|+\vartheta(|y|)$ where $\vartheta : \R_+ \to \R_+$ is a deterministic continuous increasing function.
\end{enumerate}

Now we can prove the following result.
\begin{Prop} \label{prop:existence_growth_cond_gene}
Under assumptions (H1)-(H2')-(H3) and (H4), there exists a unique $L^2$-solution $(Y,Z,\psi,M)$ for the BSDE \eqref{eq:gene_BSDE}.
\end{Prop}
\begin{proof}
The proof follows closely the arguments in \cite{klim:rozk:13} and \cite{pard:97} (see also \cite{carb:ferr:sant:07} or \cite{xia:00} for the Lipschitz case). Therefore we only sketch it.

\begin{itemize}
\item {\bf Step 1:} we assume that $f$ is Lipschitz with w.r.t.\ $y$: there exists a constant $K'$ such that for all $(t,y,y',\psi)$
\begin{equation} \label{eq:lip_assump}
|f(t,y,z,\psi) - f(t,y',z,\psi)| \leq K' |y-y'|.
\end{equation}
Moreover $\xi$ and $f(t,0,0,0)$ satisfy the condition \eqref{eq:bounded_assump}.

Under these assumptions, for $(U,V,\phi,N)$ in $\cE^2(0,T)$, we define the following processes $(Y,Z,\psi,M)$ as follows:
$$Y_t = \E \left[ \xi+ \int_0^T f(s,U_s,V_s,\phi_s) ds \bigg| \F_t \right] - \int_0^t f(s,U_s,V_s,\phi_s) ds,$$
and the local martingale
$$ \E \left[ \xi+ \int_0^T f(s,U_s,V_s,\phi_s) ds \bigg| \F_t \right] - Y_0$$
can be decomposed in three parts (see Lemma \ref{lem:decomp_loc_mart}):
$$ \E \left[ \xi+ \int_0^T f(s,U_s,V_s,\phi_s) ds \bigg| \F_t \right] - Y_0 =\int_0^t Z_s dW_s + \int_0^t \psi_s(u) \tpi(du,ds) + M_t$$
where $Z \in L^2_{loc}(W)$ a.s., $\psi \in G_{loc}(\pi)$ and $M \in \cM_{loc}$. From the conditions imposed on $f$ and $\xi$, it is straightforward to prove that $(Y,Z,\psi,M) \in \cE^2(0,T)$. Moreover $(Y,Z,\psi,M)$ is the unique solution of the BSDE
$$Y_t = \xi + \int_t^T f(s,U_s,V_s,\phi_s) ds-\int_t^T Z_s dW_s - \int_t^T \int_\cU \psi_s(u) \tpi(du,ds) - \int_t^T dM_s.$$
Therefore we may define the mapping $\Xi : \cE^2(0,T) \to \cE^2(0,T)$ by putting
$$\Xi((U,V,\phi,N)) = (Y,Z,\psi,M).$$
By standard arguments (see e.g. the proof of Theorem 55.1 in \cite{pard:97}) we can prove that $\Xi$ is contractive on the Banach space $(\cE^2(0,T),\|.\|_\beta)$ where
\begin{eqnarray*}
\|(Y,Z,\psi,M)\|_\beta & = & \E \left\{ \sup_{0\leq t\leq T} e^{\beta t} |Y_t|^2 + \int_0^T e^{\beta t} |Z_s|^2 ds \right. \\
 & & \qquad \left. +  \int_0^T e^{\beta t} \int_\cU |\psi_t(u)|^2 \mu(du)dt + \left[ \int_0^. e^{\beta t} dM_t \right]_T \right\},
 \end{eqnarray*}
with suitable constant $\beta > 0$. Consequently, $\Xi$ has a fixed point $(Y,Z,\psi,M) \in \cE^2(0,T)$. Therefore, $(Y,Z,\psi,M)$ is the unique solution of the BSDE \eqref{eq:gene_BSDE}.

\item {\bf Step 2:} We now show how to dispense with the assumptions \eqref{eq:lip_assump} and \eqref{eq:bounded_assump}. The main result is the following.
\begin{Lemma} \label{lem:tech_proof_existence}
Under assumptions ${\bf (H_{ex})}$ and (H4), given $(V,\phi)\in \bH^2(0,T) \times L^2_\pi(0,T)$ there exists a unique process $(Y,Z,\psi,M)$ in $\cE^2(0,T)$ such that
\begin{equation} \label{eq:gene_particular_BSDE}
Y_t = \xi + \int_t^T f(s,Y_s, V_s,\phi_s) ds -\int_t^T Z_s dW_s- \int_t^T\int_\cU \psi_s(u) \tpi(du,ds) - \int_t^T dM_s.
\end{equation}
\end{Lemma}
The process $f(s,y,V_s,\phi_s)$ will be denoted by $f(s,y)$.

First we keep the boundness condition \eqref{eq:bounded_assump} and we construct of smooth approximations $(f_n, \ n \in \N)$  of $f$ (see proof of Proposition 2.4 in \cite{pard:98}). For any $n$, $f_n$ is smooth and monotone in $y$, and thus locally Lipschitz in $y$. We cannot directly apply the Step 1 since $f_n$ is not necessarily globally Lipschitz. But we just add a truncation function $q_p$ in $f_n$:
$$f_{n,p}(t,y) = f_n (t,q_p(y)), \qquad q_p(y) = py/(|y|\vee p).$$
From the first step there exists a solution $(Y^{n,p},Z^{n,p},\psi^{n,p},M^{n,p})$ to BSDE \eqref{eq:gene_particular_BSDE} with generator $f_{n,p}$. Moreover from Lemma \ref{lem:Linfinity_estim}, the sequence $Y^{n,p}$ is bounded since assumption \eqref{eq:bounded_assump} holds and the upper bound on $Y^{n,p}$ does not depend on $p$. Thus for $p$ large enough, $Y^{n,p}$ does not depend on $p$, and is denoted $Y^n$ with the same on $(Z^{n,p},\psi^{n,p},M^{n,p}) = (Z^n,\psi^n,M^n)$. Now the sequence $f_n$ satisfies the assumptions of the lemmas \ref{lem:estim_mart_part} and \ref{lem:estim_Y_part} with constant independent of $n$. Thus the sequence $(Y^n,U^n,Z^n,\psi^n,M^n)$ is bounded:
$$\sup_{n \in \N} \E\left[  \int_0^T \left(  |Y^n_s|^2 +|U^n_s|^2+|Z^n_s|^2 + \int_\cU |\psi^n_s(u)|^2 \mu(du) \right) ds  + [ M^n ]_T \right] \leq C$$
where $U^n_t = f_n(t,Y^n_t)$.
Therefore there exists a subsequence which converges weakly to $(Y,U,Z,\psi,M)$. We still denote by $(Y^n,U^n,Z^n,\psi^n,M^n)$ this subsequence. The Brownian martingale $\int_.^T Z^n_s dW_s$ converges weakly in $L^2( \Omega \times [0,T])$ to $\int_.^T Z_s dW_s$ (see  \cite{pard:98}). The same trick can be applied on the Poisson martingale $\int_.^T \int_\cU \psi^n_s(u) \tpi(du,ds)$ and the orthogonal martingale $M^n$. Finally we identify $U_t$ and $f(t,Y_t)$ in the same way as in \cite{pard:98}.

Finally we remove the condition \eqref{eq:bounded_assump} by a truncation procedure. Once again we obtain a sequence $(Y^n,Z^n,\psi^n,M^n)$ which converges in $\cE^2(0,T)$ to the solution $(Y,Z,\psi,M)$ using Lemma \ref{lem:L2_stability} (see also proof of Proposition \ref{prop:Lp_estimates}).

\item {\bf Step 3:} Using the previous lemma, then we have a mapping $\Xi : \cE^2(0,T) \to \cE^2(0,T)$ which to $(U,V,\phi,N) \in \cE^2(0,T)$ associates the solution $(Y,Z,\psi,M) \in \cE^2(0,T)$ of BSDE \eqref{eq:gene_particular_BSDE}, and once again it is a contractive mapping with the norm $\|.\|_\beta$ with suitable $\beta$ (same computations as in the proof of Theorem 2.2 in \cite{pard:99}). Hence it has a fixed point $(Y,Z,\psi,M)$, solution of the BSDE \eqref{eq:gene_BSDE}.

\end{itemize}
\end{proof}

Now we are able to give the main result of this part.
\begin{Thm} \label{thm:exist_uniq_sol_gene_BSDE}
Under assumptions ${\bf (H_{ex})}$ and (H4), there exists a unique $L^2$-solution $(Y,Z,\psi,M)$ for the BSDE \eqref{eq:gene_BSDE}.
\end{Thm}
\begin{proof}
In Proposition \ref{prop:existence_growth_cond_gene} the condition (H2) was replaced by (H2'). To obtain the above result we follow the arguments of the proof of Theorem 4.2 in \cite{bria:dely:hu:03} where $f$ is approximated by a sequence of functions $f_n$ satisfying (H2') (and the other conditions). Indeed we first assume that $\xi$ and $f(t,0,0,0)$ are bounded. We can construct two sequences $f_n$ and $h_n$ satisfying ${\bf (H_{ex})}$ as in \cite{bria:dely:hu:03}. To be more precise, let 
\begin{itemize}
\item $\theta_r$ be a smooth function such that $0\leq \theta_r \leq 1$, $\theta_r(y) =1$ if $|y|\leq r$ and $\theta_r(y) = 0$ as soon as $|y|\geq r+1$;
\item $\rho_r(t) = \sup_{|y|\leq r} |f(t,y,0,0)-f(t,0,0,0)|$;
\item for each $n\in \N^*$, $q_n(z) = zn/(|z|\vee n)$.
\end{itemize}
Then
\begin{eqnarray*}
f_n(t,y,z,\psi) & =& \left[ f(t,y,q_n(z),q_n(\psi))-f(t,0,0,0)\right] \frac{n}{\rho_{r+1}(t)\vee n} + f(t,0,0,0),\\
h_n(t,y,z,\psi) &=& \theta_r(y) \left[ f(t,y,q_n(z),q_n(\psi))-f(t,0,0,0)\right] \frac{n}{\rho_{r+1}(t)\vee n} + f(t,0,0,0).
\end{eqnarray*}
Note that we also truncate the part on $\psi$ in $f_n$ and $h_n$ truncates $f_n$ for $|y| \geq r+1$. 

For fixed $r$ and $n$, $h_n$ satisfies the conditions of Proposition \ref{prop:existence_growth_cond_gene}. Hence there exists a unique solution $(Y^n,Z^n,\psi^n,M^n)$ in $\cE^2(0,T)$ with generator $h_n$ and from Lemma \ref{lem:Linfinity_estim}, $Y^n$ satisfies the inequality $\|Y^n\|_{\infty} \leq r$. Lemma \ref{lem:estim_mart_part} shows that
\begin{equation} \label{eq:L2_estim_gene_case}
\E\left(  \int_0^T |Z^n_s|^2 ds + \int_0^T \int_\cU | \psi^n_s(u)|^2 \mu(du) ds + [M^n ]_T \right) \leq r'.
\end{equation}
Therefore if we have chosen $r$ large enough, $(Y^n,Z^n,\psi^n,M^m)$ is solution of the BSDE \eqref{eq:gene_BSDE} with generator $f_n$ satisfying ${\bf (H_{ex})}$.
By It\^o's formula on $U = Y^{n+i}-Y^n$, $V=Z^{n+i}-Z^n$, $\phi = \psi^{n+i} - \psi^n$, $N =M^{n+i} - M^n$:
\begin{eqnarray*}
&& e^{4K^2t} |U_t|^2  + \frac{1}{2} \int_{t}^{T} e^{4K^2s} |V_s|^2 ds + \frac{1}{2} \int_{t}^{T} \int_\cU e^{4K^2s} |\phi_s(u)|^2 \mu(du) ds +  \int_t^T e^{4K^2s} d[ N]_s \\
&& \leq 2 \int_{t}^{T} e^{4K^2s} U_s (f_{n+i}(s,Y^{n}_s,Z^{n}_s,\psi^{n}_s)-f_{n}(s,Y^{n}_s,Z^{n}_s,\psi^{n}_s)) ds - 2 \int_{t}^{T} e^{4K^2s} U_{s} V_s dW_s\\
&&\qquad  -2 \int_{t}^{T} e^{4K^2s}  U_{s^-} dN_s  -\int_{t}^{T}e^{4K^2s}  \int_\cU \left( |U_{s^-} +\phi_s(u) |^2 - |U_{s^-}|^2 \right) \tpi(du,ds) \\
&& \leq 4r \int_{t}^{T} e^{4K^2s} |f_{n+i}(s,Y^{n}_s,Z^{n}_s,\psi^{n}_s)-f_{n}(s,Y^{n}_s,Z^{n}_s,\psi^{n}_s)| ds - 2 \int_{t}^{T} e^{4K^2s} U_{s} V_s dW_s\\
&&\qquad  -2 \int_{t}^{T} e^{4K^2s}  U_{s^-} dN_s  -\int_{t}^{T}e^{4K^2s}  \int_\cU \left( |U_{s^-} +\phi_s(u) |^2 - |U_{s^-}|^2 \right) \tpi(du,ds)
\end{eqnarray*}
since $\|U\|_\infty \leq 2r$. Using the Burkholder-Davis-Gundy (BDG) inequality we get for some constant depending on $K$ and $T$:
\begin{eqnarray*}
&& \E \left(   \sup_{t\in [0,T]} |U_t|^2  +\int_{0}^{T}  |V_s|^2 ds +  \int_{0}^{T} \int_\cU |\phi_s(u)|^2 \mu(du) ds + [ N]_T \right) \\
&& \qquad \leq Cr \E  \int_{0}^{T}  |f_{n+i}(s,Y^{n}_s,Z^{n}_s,\psi^{n}_s)-f_{n}(s,Y^{n}_s,Z^{n}_s,\psi^{n}_s)| ds. 
\end{eqnarray*}
Since $\|Y^n\|_\infty \leq r$, from the definition of $f_n$ we have
\begin{eqnarray*}
&& |f_{n+i}(s,Y^{n}_s,Z^{n}_s,\psi^{n}_s)-f_{n}(s,Y^{n}_s,Z^{n}_s,\psi^{n}_s)| \\
&&\qquad  \leq 2K |Z^n_s| \1_{|Z^n_s| > n} +2K \|\psi^n_s\|_{L^2} \1_{\|\psi^n_s\|_{L^2} > n}+2K |Z^n_s| \1_{\rho_{r+1}(s) > n} \\
&& \qquad \qquad +2K \|\psi^n_s\|_{L^2} \1_{\rho_{r+1}(s) > n} + 2\rho_{r+1}(s) \1_{\rho_{r+1}(s) > n}  .
\end{eqnarray*}
Since $\rho_{r+1} \in L^1(\Omega \times [0,T], \P \otimes m)$ (Assumption (H2)), and $(Z^n,\psi^n) \in \bH^2(0,T) \times L^2_\mu$ uniformly w.r.t. $n$ (Inequality \eqref{eq:L2_estim_gene_case}), this implies that $(Y^n,Z^n,\psi^n,M^m)$ is a Cauchy sequence in $\cE^2(0,T)$.

The general case will be obtained by a truncation procedure on $\xi$ and $f(t,0,0,0)$ and the inequality of Lemma \ref{lem:L2_stability}.
\end{proof}

\section{Existence in $L^p$}

The following proposition was proved in the Lipschitz case without jumps in \cite{elka:peng:quen:97}, Section 5, or in \cite{bria:dely:hu:03} for the Brownian filtration, for any $p > 1$.
\begin{Prop}[$L^p$-estimates, $p\geq 2$] \label{prop:Lp_estimates}
We assume that $f$ satisfies ${\bf (H_{ex})}$. For $p\geq 2$, if we have
\begin{equation} \tag{H5}
 \E \left(|\xi|^p +  \int_0^T |f(t,0,0,0)|^p dt \right) < +\infty,
 \end{equation}
then the solution $(Y,Z,\psi,M)$ belongs to $\cE^p(0,T)$. Moreover there exists a constant $C$ depending only on $K^2$, $p$ and $T$ such that
\begin{eqnarray}\nonumber
&& \E\left[ \sup_{t\in[0,T]} |Y_t|^p +  \left( \int_0^T |Z_s|^2 ds \right)^{p/2} + \left( \int_0^T \int_\cU |\psi_s(u)|^2 \mu(du) ds \right)^{p/2} + [M]_T^{p/2} \right] \\ \label{eq:Lp_estim_sol_p_geq 2}
&& \qquad\qquad  \leq C \E \left( | \xi|^{p} + \int_0^T |f(s,0,0,0)|^p ds \right). 
\end{eqnarray}
\end{Prop}

\begin{proof}
Under this condition (H5) on $\xi$ and $f(t,0,0,0)$, we know that there exists a unique solution $(Y,Z,\psi,M)$ which belongs to $\cE^2(0,T)$. We want to show that $(Y,Z,\psi,M)$ in fact belongs to $\cE^p(0,T)$.

From the proof of Theorem \ref{thm:exist_uniq_sol_gene_BSDE} (or Proposition \ref{prop:existence_growth_cond_gene}), the solution $(Y,Z,\psi,M)$ is obtained as the limit of a sequence $(Y^n,Z^n,\psi^n,M^n)$, solution of BSDE \eqref{eq:gene_BSDE} but with bounded coefficients $\xi^n$  and $f^n(t,0,0,0)$. We prove that convergence also holds in $\cE^p(0,T)$ by proving the counterpart of Lemma \ref{lem:L2_stability} in $\cE^p$. For any $(m,n)\in \N^2$ we denote
$$(\what Y,\what Z, \what \psi, \what M, \what \xi, \what f) = (Y^m-Y^n,Z^m-Z^n,\psi^m-\psi^n, M^m-M^n, \xi^m - \xi^n, f^m-f^n).$$

\noindent {\bf Step 1:} we prove that the sequence $Y^n$ converges in $\bD^p(0,T)$ to $Y$. Since $p \geq 2$ we can apply It\^o formula with the $C^2$-function $\theta(y) = |y|^p$ to the process $\what Y$. Note that
$$\frac{\partial \theta}{\partial y_i} (y)= p y_i |y|^{p-2}, \frac{\partial^2 \theta}{\partial y_i \partial y_j} (y)= p |y|^{p-2} \delta_{i,j} + p(p-2)y_iy_j|y|^{p-4}$$
where $\delta_{i,j}$ is the Kronecker delta. Therefore for every  $0\leq t \leq T$ we have:
\begin{eqnarray} \nonumber
&& |\what Y_{t}|^p = |\what \xi|^p+ \int_{t}^{T} p \what Y_s |\what Y_s|^{p-2} (f^m(s,Y^m_s,Z^m_s,\psi^m_s) - f^n(s,Y^n_s,Z^n_s,\psi^n_s)) ds \\  \nonumber
& &\quad  -  p  \int_{t}^{T} \what Y_{s^-} |\what Y_{s^-}|^{p-2} d\what M_s  -  p  \int_{t}^{T} \what Y_{s^-} |\what Y_{s^-}|^{p-2} \what Z_s dW_s    \\ \nonumber
& &\quad  - p \int_{t}^{T} \int_\cU \left( \what  Y_{s^-} |\what Y_{s^-}|^{p-2} \what \psi_s(u) \right) \tpi(du,ds) - \frac{1}{2}\int_t^T \trace \left( D^2\theta(\what Y_s) \what Z_s \what Z_s^* \right) ds \\ \nonumber
& &\quad - \int_{t}^{T}  \int_\cU \left( |\what Y_{s^-} + \what \psi_s(u) |^p - |\what Y_{s^-}|^p - p\what  Y_{s^-} |\what Y_{s^-}|^{p-2} \what \psi_s(u) \right) \pi(du,ds) \\  \nonumber
&& \quad - \frac{1}{2} \int_t^T \sum_{1\leq i,j\leq d}  \frac{\partial^2 \theta}{\partial y_i \partial y_j} (\what Y_s) d[ \what M^i, \what M^j ]^c_t \\ \label{eq:Ito_formula_3}
&& \quad - \sum_{t< s \leq T} \left( |\what Y_{s^-} + \Delta \what M_s |^p - |\what Y_{s^-}|^p - p\what  Y_{s^-} |\what Y_{s^-}|^{p-2}\Delta \what M_s \right).
\end{eqnarray}
The notation $[ M ]^c$ denotes the continuous part of the bracket process $[ M ]$. First remark that for a non negative symmetric matrix $\Gamma \in \R^{d\times d}$
\begin{equation*}
\sum_{1\leq i,j \leq d} D^2\theta(y)_{i,j} \Gamma_{i,j} = p |y|^{p-2} \trace(\Gamma) + p(p-2) |y|^{p-4} (y^*)\Gamma y \geq p |y|^{p-2} \trace(\Gamma) ,
\end{equation*}
and thus
$$\trace (D^2\theta(y) z z^*)   \geq p |y|^{p-2} |z|^2.$$
Moreover using Taylor formula (and Lemma A.4 in \cite{yao:10} for the last inequality) we have
\begin{eqnarray*}
&& \theta(x+y) - \theta(x) - \nabla \theta (x) y = \int_0^1 y D^2\theta(x+r y) y (1-r) dr \\
&&\qquad = p|y|^2 \int_0^1(1-r) |x+ry|^{p-2} dr + p(p-2) \int_0^1 (y(x+r y))^2 |x+r y|^{p-4}(1-r) dr \\
&& \qquad \geq p|y|^2 \int_0^1(1-r) |x+ry|^{p-2} dr \geq  p(p-1) 3^{1-p} |y|^2|x|^{p-2}.
\end{eqnarray*}
Therefore we deduce that
\begin{eqnarray*}
&& \frac{1}{2} \int_t^T \sum_{1\leq i,j\leq d}  \frac{\partial^2 \theta}{\partial y_i \partial y_j} (\what Y_s) d[ \what M^i,  \what M^j ]^c_t  + \sum_{t< s \leq T} \left( |\what Y_{s^-} + \Delta \what M_s |^p - |\what Y_{s^-}|^p - p\what  Y_{s^-} |\what Y_{s^-}|^{p-2}\Delta \what M_s(z) \right) \\
&& \qquad \geq \frac{p}{2}\int_t^T  | \what Y_s|^{p-2} d[ \what M]^c_s +  p(p-1) 3^{1-p} \sum_{t< s \leq T} | \what Y_s|^{p-2} |\Delta \what M_s|^2 \geq \kappa_p \int_t^T  | \what Y_s|^{p-2} d[  \what M]_s
\end{eqnarray*}
where $\kappa_p = \min (p/2,p(p-1) 3^{1-p} ) > 0$. Now the Poisson part in \eqref{eq:Ito_formula_3} can be written as follows:
\begin{eqnarray*}
& &- p \int_t^T \int_\cU \left( \what  Y_{s^-} |\what Y_{s^-}|^{p-2} \what \psi_s(u) \right) \tpi(du,ds)  \\ \nonumber
& &\qquad - \int_t^T   \int_\cU \left( |\what Y_{s^-} + \what \psi_s(u) |^p - |\what Y_{s^-}|^p - p\what  Y_{s^-} |\what Y_{s^-}|^{p-2} \what \psi_s(u) \right) \pi(du,ds) \\  \nonumber
& &\ = - \int_t^T  \int_\cU \left( |\what Y_{s^-} + \what \psi_s(u) |^p - |\what Y_{s^-}|^p - p\what  Y_{s^-} |\what Y_{s^-}|^{p-2} \what \psi_s(u) \right) \mu(du) ds \\
& &\qquad - \int_t^T \int_\cU \left( |\what Y_{s^-} + \what \psi_s(u) |^p - |\what Y_{s^-}|^p \right) \tpi(du,ds) \\
&& \leq -  p(p-1) 3^{1-p} \int_t^T  |\what Y_{s^-}  |^{p-2}  \|\what \psi_s\|_{L^2_\mu}^2 ds  - \int_t^T \int_\cU \left( |\what Y_{s^-} + \what \psi_s(u) |^p - |\what Y_{s^-}|^p \right) \tpi(du,ds) .
\end{eqnarray*}
Then \eqref{eq:Ito_formula_3} becomes
\begin{eqnarray*}
&& |\what Y_{t}|^p +  \kappa_p \int_t^T |\what Y_{s}|^{p-2} | \what Z_s|^2 ds + \kappa_p \int_t^T |\what Y_{s^-}|^{p-2} d[ \what M ]_s + \kappa_p\int_t^T  |\what Y_{s^-}  |^{p-2}  \|\what \psi_s\|_{L^2_\mu}^2 ds \\
&&\quad \leq  |\what \xi|^p  + \int_{t}^{T} p \what Y_s |\what Y_s|^{p-2} (f^m(s,Y^m_s,Z^m_s,\psi^m_s) - f^n(s,Y^n_s,Z^n_s,\psi^n_s)) ds \\
& &\quad  -  p  \int_{t}^{T} \what Y_{s^-} |\what Y_{s^-}|^{p-2} d\what M_s  -  p  \int_{t}^{T} \what Y_{s^-} |\what Y_{s^-}|^{p-2} \what Z_s dW_s    \\
&& \quad  - \int_t^T \int_\cU \left( |\what Y_{s^-} + \what \psi_s(u) |^p - |\what Y_{s^-}|^p \right) \tpi(du,ds).
\end{eqnarray*}
From the assumptions on $f^m$, we still have \eqref{eq:young_growth_f} and we choose $\eps = \frac{\kappa_p}{p}$. We obtain
\begin{eqnarray*}
&& |\what Y_{t}|^p + \frac{ \kappa_p}{2} \int_t^T |\what Y_{s}|^{p-2} | \what Z_s|^2 ds+  \frac{\kappa_p}{2} \int_t^T |\what Y_{s^-}|^{p-2} d[\what M ]_s + \frac{\kappa_p}{2} \int_t^T  |\what Y_{s^-}  |^{p-2}  \|\what \psi_s\|_{L^2_\mu}^2 ds \\
&&\quad \leq  |\what \xi|^p  + \frac{p^2}{2\kappa_p}(2K^2+1)\int_t^T  |\what Y_s|^{p} ds  +\frac{\kappa_p}{2} \int_t^T|\what Y_s|^{p-2} |\what f(s,Y^n_s,Z^n_s,\psi^n_s)|^2 ds \\
& &\quad  -  p  \int_{t}^{T} \what Y_{s^-} |\what Y_{s^-}|^{p-2} d\what M_s  -  p  \int_{t}^{T} \what Y_{s^-} |\what Y_{s^-}|^{p-2} \what Z_s dW_s    \\
&& \quad  - \int_t^T \int_\cU \left( |\what Y_{s^-} + \what \psi_s(u) |^p - |\what Y_{s^-}|^p \right) \tpi(du,ds).
\end{eqnarray*}
Using Young's inequality, we finally have
\begin{eqnarray} \nonumber
&& |\what Y_{t}|^p + \frac{ \kappa_p}{2} \int_t^T |\what Y_{s}|^{p-2} | \what Z_s|^2 ds+  \frac{\kappa_p}{2} \int_t^T |\what Y_{s^-}|^{p-2} d[ \what M ]_s + \frac{\kappa_p}{2} \int_t^T  |\what Y_{s^-}  |^{p-2}  \|\what \psi_s\|_{L^2_\mu}^2 ds \\ \nonumber
&&\quad \leq  |\what \xi|^p  +\left[  \frac{p^2}{2\kappa_p}(2K^2+1) + \frac{\kappa_p(p-2)}{2p}\right] \int_t^T |\what Y_s|^{p} ds  +\frac{\kappa_p }{p} \int_t^T   |\what f(s,Y^n_s,Z^n_s,\psi^n_s)|^p ds \\ \nonumber
& &\quad  -  p  \int_{t}^{T} \what Y_{s^-} |\what Y_{s^-}|^{p-2} d\what M_s  -  p  \int_{t}^{T} \what Y_{s^-} |\what Y_{s^-}|^{p-2} \what Z_s dW_s    \\  \label{eq:Lp_estimate_1}
&& \quad  - \int_t^T \int_\cU \left( |\what Y_{s^-} + \what \psi_s(u) |^p - |\what Y_{s^-}|^p \right) \tpi(du,ds).
\end{eqnarray}
Note that the three local martingales in the previous inequality are true martingales. Indeed since $Y^m$ and $Y^n$ are in $\bD^\infty(0,T)$ and $M^m$ and $M^n$ are in $\bM^2(0,T)$, the local martingale
$$ \int_{0}^{.} \what Y_{s^-} |\what Y_{s^-}|^{p-2} d\what M_s$$
is a true martingale and we can apply the Burkholder-Davis-Gundy inequality to obtain
\begin{eqnarray} \nonumber
&& \E \sup_{t\in [0,T]} \left|  \int_{0}^{t} \what Y_{s^-} |\what Y_{s^-}|^{p-2} d\what M_s \right|  \leq c_p \E \left(   \int_{0}^{T} |\what Y_{s^-}|^{2p-2} d[ \what M ]_s \right)^{1/2} \\ \label{eq:Lp_estimate_2}
&& \qquad \qquad \leq \frac{1}{4p} \E \left( \sup_{t \in [0,T]}  |\what Y_{t}|^{p}  \right) + p c_p^2 \E  \int_{0}^{T} |\what Y_{s^-}|^{p-2} d[ \what M ]_s.
\end{eqnarray}
By the same arguments we have
\begin{eqnarray} \nonumber
&& \E \sup_{t\in [0,T]} \left|  \int_{0}^{t} \what Y_{s^-} |\what Y_{s^-}|^{p-2} \what Z_s dW_s \right|  \leq c_p \E \left(   \int_{0}^{T} |\what Y_{s^-}|^{2p-2} | \what Z_s|^2 ds \right)^{1/2} \\  \label{eq:Lp_estimate_3}
&& \qquad \qquad \leq \frac{1}{4p} \E \left( \sup_{t \in [0,T]}  |\what Y_{t}|^{p}  \right) + p c_p^2 \E  \int_{0}^{T} |\what Y_{s}|^{p-2} |\what Z_s|^2 ds .
\end{eqnarray}
Finally the same result holds for the martingale
$$\int_{0}^{.} \int_\cU \left( \what  Y_{s^-} |\what Y_{s^-}|^{p-2} \what \psi_s(u) \right) \tpi(du,ds),$$
with
\begin{eqnarray} \nonumber
&& \E \sup_{t\in [0,T]} \left| \int_{0}^{t} \int_\cU \left( \what  Y_{s^-} |\what Y_{s^-}|^{p-2} \what \psi_s(u) \right) \tpi(du,ds) \right|  \leq c_p \E \left(   \int_{0}^{T} |\what Y_{s^-}|^{2p-2}\int_\cU |\what \psi_s(u)|^2 \pi(du,ds) \right)^{1/2} \\  \label{eq:Lp_estimate_4}
&& \qquad \qquad \leq \frac{1}{4p} \E \left( \sup_{t \in [0,T]}  |\what Y_{t}|^{p}  \right) +  pc_p^2 \E  \int_{0}^{T} |\what Y_{s^-}|^{p-2}  \|\what \psi_s\|_{L^2_\mu}^2 ds.
\end{eqnarray}

Now we come to the conclusion. Using \eqref{eq:Lp_estimate_1} we can take expectations and obtain for every $0\leq t \leq T$:
\begin{eqnarray*}
&&\E |\what Y_{t}|^p  \leq \E |\what \xi|^p  +\left[  \frac{p^2}{2\kappa_p}(2K^2+1) + \frac{\kappa_p(p-2)}{2p}\right] \E  \int_t^T |\what Y_s|^{p} ds  +\frac{\kappa_p }{p}\E  \int_t^T   |\what f(s,Y^n_s,Z^n_s,\psi^n_s)|^p ds,
\end{eqnarray*}
hence by Gronwall's lemma
\begin{eqnarray*}
&&\E |\what Y_{t}|^p  \leq C  \E \left( |\what \xi|^p+\E \int_{0}^{T}   |\what f(s,Y^n_s,Z^n_s,\psi^n_s)|^p ds\right)
\end{eqnarray*}
for some constant $C$ depending on $K$, $p$ and $T$. From this and \eqref{eq:Lp_estimate_1} again we also deduce that
\begin{eqnarray*}
&& \E \int_0^T |\what Y_{s}|^p ds + \E \int_{0}^{T} |\what Y_{s}|^{p-2} |Z_s|^2 ds  + \E \int_{0}^{T} |\what Y_{s^-}|^{p-2} d[ \what M ]_s + \E \int_{0}^{T}  |\what Y_{s^-}  |^{p-2}  \|\what \psi_s\|_{L^2_\mu}^2 ds \\
&& \qquad \qquad  \leq C  \E \left( |\what \xi|^p+\E \int_{0}^{T}   |\what f(s,Y^n_s,Z^n_s,\psi^n_s)|^p ds\right).
\end{eqnarray*}

Let us come back to \eqref{eq:Ito_formula_3} and use the convexity of the function $\theta$ and Estimate \eqref{eq:young_growth_f} with $\eps = \kappa_p/p$, to deduce that:
\begin{eqnarray} \nonumber
&& |\what Y_{t}|^p \leq |\what \xi|^p+\left[  \frac{p^2}{2\kappa_p}(2K^2+1) + \frac{\kappa_p(p-2)}{2p}\right] \int_t^T |\what Y_s|^{p} ds  +\frac{\kappa_p }{p} \int_t^T   |\what f(s,Y^n_s,Z^n_s,\psi^n_s)|^p ds \\  \nonumber
&&\quad +  \frac{ \kappa_p}{2} \int_t^T |\what Y_{s}|^{p-2} | \what Z_s|^2 ds+ \frac{\kappa_p}{2} \int_t^T  |\what Y_{s^-}  |^{p-2}  \|\what \psi_s\|_{L^2_\mu}^2 ds\\ \nonumber
& &\quad  -  p  \int_{t}^{T} \what Y_{s^-} |\what Y_{s^-}|^{p-2} d\what M_s  -  p  \int_{t}^{T} \what Y_{s^-} |\what Y_{s^-}|^{p-2} \what Z_s dW_s    \\ \nonumber
& &\quad  - p \int_{t}^{T} \int_\cU \left( \what  Y_{s^-} |\what Y_{s^-}|^{p-2} \what \psi_s(u) \right) \tpi(du,ds) .
\end{eqnarray}
Now using estimates \eqref{eq:Lp_estimate_2}, \eqref{eq:Lp_estimate_3} and \eqref{eq:Lp_estimate_4},  we get:
$$\E \left( \sup_{t \in [0,T]}  |\what Y_{t}|^{p}  \right)\leq C  \E \left( |\what \xi|^p+\E \int_{0}^{T}   |\what f(s,Y^n_s,Z^n_s,\psi^n_s)|^p ds\right).$$
Therefore the limit process $Y$ belongs to $\bD^p(0,T)$. 

\noindent {\bf Step 2:} We adopt the arguments of the proof of Lemma \ref{lem:estim_mart_part} (see also Lemma 3.1 in \cite{bria:dely:hu:03}) to prove that:
\begin{eqnarray} \nonumber 
&& \E\left[  \left( \int_0^T |\what Z_s|^2 ds \right)^{p/2} + \left( \int_0^T \int_\cU |\what \psi_s(u)|^2 \mu(du) ds \right)^{p/2} +[ \what M ]_T^{p/2} \right] \\ \label{eq:Lp_main_estim_mart_part_p_geq_2}
&& \qquad \leq C \E \left[ \sup_{t \in [0,T]} | \what Y_{t}|^{p} + \int_0^T |\what f(s,Y^n_s,Z^n_s,\psi^n_s)|^p ds \right].
\end{eqnarray}
This estimate gives the convergence of $(Z^n,\psi^n,M^n)$ in the desired integrability space. Indeed let $\tau_k \in \cT_T$ defined by:
$$\tau_k = \inf \left\{ t \in [0,T], \int_0^t |\what Z_r|^2 dr +\int_0^t \int_\cU |\what \psi_s(u)|^2 \pi(du,ds) + [ \what M ]_t \geq k \right\} \wedge 
T.$$
By It\^o's formula on $|\what Y_t|^2$: 
\begin{eqnarray*} 
&& |\what Y_{0}|^2 + \int_{0}^{\tau_k} |\what Z_s|^2 ds + \int_{0}^{\tau_k} \int_\cU |\what \psi_s(u)|^2 \pi(du,ds) + [\what  M]_{\tau_k} \\ 
&& \quad =  |\what Y_{\tau_k}|^2 + 2 \int_{0}^{\tau_k} \what Y_s \left( f^m(s,Y^m_s,Z^m_s,\psi^m_s) -  f^n(s,Y^n_s,Z^n_s,\psi^n_s) \right) ds \\
&& \qquad -2 \int_{0}^{\tau_k} \what Y_{s^-} \what Z_s dW_s  -2 \int_{0}^{\tau_k} \what Y_{s^-} d\what M_s  - 2 \int_{0}^{\tau_k} \int_\cU \what Y_{s^-} \what \psi_s(u)  \tpi(du,ds) .
\end{eqnarray*}
Once again with a straightforward modification of estimate \eqref{eq:young_growth_f}:
\begin{eqnarray*} \nonumber
&& \frac{1}{2} \int_{0}^{\tau_k} |\what Z_s|^2 ds +\int_{0}^{\tau_k} \int_\cU |\what \psi_s(u)|^2 \pi(du, ds) + [\what  M]_{\tau_k} \\  \nonumber
&& \quad \leq |\what Y_*|^2 + \frac{((1+1/\eps)K^2+1)}{2} \int_{0}^{T} |\what Y_s|^2 ds + \frac{1}{2} \int_{0}^{T} |\what f(s,Y^n_s,Z^n_s,\psi^n_s)|^2 ds   \\ \nonumber
&& \qquad + \frac{\eps}{2} \int_{0}^{\tau_k} \int_\cU |\what \psi_s(u)|^2 \mu(du) ds \\ 
&& \qquad -2 \int_{0}^{\tau_k} \what Y_{s^-} \what Z_s dW_s  -2 \int_{0}^{\tau_k} \what Y_{s^-} d\what M_s  - 2 \int_{0}^{\tau_k} \int_\cU \what Y_{s^-} \what \psi_s(u) \tpi(du,ds)
\end{eqnarray*}
where $\what Y_* =\sup_{t\in[0,T]} |\what Y_t|$. It follows that 
\begin{eqnarray}   \nonumber
&& \left( \int_{0}^{\tau_k} |\what Z_s|^2 ds \right)^{p/2}+\left(  \int_{0}^{\tau_k} \int_\cU |\what \psi_s(u)|^2 \pi(du, ds)  \right)^{p/2}+ [\what  M]^{p/2}_{\tau_k} \\   \nonumber
&& \quad \leq C_p \left[ \left( 1 + \frac{((1+1/\eps)K^2+1)T}{2} \right)^{p/2} |\what Y_*|^p  +\left(  \int_{0}^{T} |\what f(s,Y^n_s,Z^n_s,\psi^n_s)|^2 ds \right)^{p/2}\right]  \\  \nonumber
&& \qquad + C_p \eps^{p/2} \left(  \int_{0}^{\tau_k} \int_\cU |\what \psi_s(u)|^2 \mu(du) ds \right)^{p/2} + C_p \left|\int_{0}^{\tau_k} \int_\cU \what Y_{s^-} \what \psi_s(u)  \tpi(du,ds)\right|^{p/2}\\ \label{eq:Lp_estim_mart_part_p_geq_2}
&& \qquad + C_p \left[ \left| \int_{0}^{\tau_k} \what Y_{s^-} \what Z_s dW_s \right|^{p/2} + \left| \int_{0}^{\tau_k} \what Y_{s^-} d\what M_s \right|^{p/2}  \right].
\end{eqnarray}
Since $p/2 \geq 1$, we can apply the Burkholder-Davis-Gundy inequality to obtain
\begin{eqnarray*} \nonumber
&& C_p \E  \left|  \int_{0}^{\tau_k} \what Y_{s^-} d\what M_s \right|  \leq  d_p \E\left[ \left(   \int_{0}^{\tau_k} |\what Y_{s^-}|^{2}  d[ \what M]_s \right)^{p/4} \right] \leq \frac{d_p^2}{4} \E \left(  |\what Y_{*}|^{p}  \right) + \frac{1}{2}  [ \what M]_{\tau_k}^{p/2},\\
&& C_p \E \left|  \int_{0}^{\tau_k} \what Y_{s^-}  \what Z_s dW_s \right| \leq d_p \E\left[ \left(   \int_{0}^{\tau_k} |\what Y_{s^-}|^{2} | \what Z_s|^2 ds \right)^{p/4} \right]\\
&&\qquad \leq  \frac{d_p^2}{4} \E \left(  |\what Y_{*}|^{p}  \right) + \frac{1}{2} \E \left[ \left(\int_{0}^{\tau_k}  |\what Z_s|^2 ds\right)^{p/2} \right],\\
&& C_p \E\left| \int_{0}^{\tau_k} \int_\cU \what  Y_{s^-}  \what \psi_s(u) \tpi(du,ds) \right|  \leq  d_p \E \left[ \left(   \int_{0}^{\tau_k} |\what Y_{s^-}|^{2}\int_\cU |\what \psi_s(u)|^2 \pi(du,ds) \right)^{p/4}\right] \\ 
 &&\qquad \leq  \frac{d_p^2}{4} \E \left(   |\what Y_{*}|^{p}  \right) +  \frac{1}{2} \E  \left[ \left(\int_{0}^{\tau_k}   \|\what \psi_s\|_{L^2_\mu}^2 \pi(du,ds) \right)^{p/2}\right].
\end{eqnarray*}
Hence coming back to \eqref{eq:Lp_estim_mart_part_p_geq_2} and taking the expectation 
\begin{eqnarray*}  
&&\frac{1}{2} \E \left( \int_{0}^{\tau_k} |\what Z_s|^2 ds \right)^{p/2}+\frac{1}{2} \E\left(  \int_{0}^{\tau_k} \int_\cU |\what \psi_s(u)|^2 \pi(du, ds)  \right)^{p/2}+\frac{1}{2} \E [\what  M]^{p/2}_{\tau_k} \\   
&& \quad \leq C_{p,K,T,\eps} \ \E   |\what Y_*|^p  + C_p \E \left[ \left(  \int_{0}^{T} |\what f(s,Y^n_s,Z^n_s,\psi^n_s)|^2 ds \right)^{p/2}\right]  \\  \nonumber
&& \qquad + C_p \eps^{p/2}  \E \left(  \int_{0}^{\tau_k} \int_\cU |\what \psi_s(u)|^2 \mu(du) ds \right)^{p/2} .
\end{eqnarray*}
Finally we use that for some constant $e_p>0$  
\begin{equation} \label{eq:pred_quad_var}
 \E \left(  \int_{0}^{\tau_k} \int_\cU |\what \psi_s(u)|^2 \mu(du) ds \right)^{p/2} \leq e_p \E\left(  \int_{0}^{\tau_k} \int_\cU |\what \psi_s(u)|^2 \pi(du, ds)  \right)^{p/2},
 \end{equation}
(see \cite{leng:lepi:prat:80,dzha:valk:90}) and thus we can choose $\eps$ sufficiently small and depending only on $p$ such that:
\begin{eqnarray}  \nonumber
&& \E \left( \int_{0}^{\tau_k} |\what Z_s|^2 ds \right)^{p/2}+ \E\left(  \int_{0}^{\tau_k} \int_\cU |\what \psi_s(u)|^2 \mu(du) ds  \right)^{p/2}+ \E [\what  M]^{p/2}_{\tau_k} \\   \label{eq:Lp_estim_mart_part_p_geq _2}
&& \quad \leq \widetilde C_{p,K,T,\eps} \ \E   |\what Y_*|^p  + \widetilde C_p \E \left[ \left(  \int_{0}^{T} |\what f(s,Y^n_s,Z^n_s,\psi^n_s)|^2 ds \right)^{p/2}\right]   .
\end{eqnarray}
We can let $k$ go to $+\infty$ in order to have estimate \eqref{eq:Lp_main_estim_mart_part_p_geq_2}.

\noindent {\bf Step 3:} 
The inequality \eqref{eq:Lp_estim_sol_p_geq 2} can be deduced from the previous steps: we just replace $(Y^m,Z^m,\psi^m,M^m,f^m)$ by $(Y,Z,\psi,M,f)$ and $(Y^n,Z^n,\psi^n,M^n,f)$ by $(0,0,0,0,0)$. 
\end{proof}

Now we consider the case where $p\in [1,2)$. The main difference is that we cannot directly apply the It\^o formula to $\theta(y) = |y|^p$. The next result is an extension of the Meyer-It\^o formula and as mentioned in \cite{bria:dely:hu:03}, it is likely that this result already appeared somewhere. A version of this result is given in Lemma 2.2 in \cite{bria:dely:hu:03} without jumps or in Proposition 2.1 in \cite{klim:13b} in dimension one. We denote by $\check{x} = |x|^{-1} x \1_{x\neq 0}$.
\begin{Lemma} \label{lem:Ito_formula_x_p}
We consider the $\R^d$-valued semimartingale $(X_t)_{t\in [0,T]}$ defined by
$$X_t = X_0 + \int_0^t K_s ds + \int_0^t Z_s dW_s + \int_0^t \int_\cU \psi_s(u) \tpi(du,ds) + M_t,$$
such that $t\mapsto K_t$ belongs to $L^1_{loc}(0,+\infty)$ a.s., $Z \in L^2_{loc}(W)$, $\psi \in G_{loc}(\pi)$ and $M \in \cM_{loc}$.
Then for any $p\geq 1$, we have
\begin{eqnarray} \nonumber
&& |X_t|^p = |X_0|^p + \frac{1}{2} L(t) \1_{p=1} + p \int_0^t |X_s|^{p-1} \check{X}_s K_s ds  +  p \int_0^t |X_s|^{p-1} \check{X}_s Z_s dW_s  \\ \nonumber
& &\quad  +  p  \int_0^t |X_{s^-}|^{p-1} \check{X}_{s^-} dM_s +  p \int_0^t |X_s|^{p-1} \check{X}_s  \int_\cU \psi_s(u) \tpi(du,ds)   \\ \nonumber
& &\quad +  \int_{0}^{t}  \int_\cU \left[ |X_{s^-}+\psi_s(u)|^p -|X_{s^-}|^p - p|X_{s^-}|^{p-1} \check{X}_{s^-} \psi_s(u) \right] \pi(du,ds) \\  \nonumber
&& \quad +\sum_{0< s \leq t}  \left[ |X_{s^-}+\Delta M_s|^p - |X_{s^-}|^p - p|X_{s^-}|^{p-1} \check{X}_{s^-} \Delta M_s \right] \\ \nonumber
&&\quad + \frac{p}{2}\int_0^t  |X_{s}|^{p-2}  \1_{X_s\neq 0}\left\{(2-p)\left[ |Z_s|^2 - (\check{X_{s}})^* Z_s Z_s^* \check{X_{s}} \right] + (p-1)|Z_s|^2 \right\} ds \\ \label{eq:Ito_formula_x_p}
&& \quad +\frac{p}{2} \int_0^t  |X_{s}|^{p-2}  \1_{X_s\neq 0}  \left\{(2-p)\left[ d[ M ]^c_s - (\check{X_{s}})^* d[M,M]^c_s \check{X_{s}} \right] + (p-1)d[ M ]^c_s \right\} .
\end{eqnarray}
The process $ (L(t),\ t\in[0; T ])$ is continuous, nondecreasing with $L_0 = 0$ and increases only on the boundary of the random set $\{t\in[0;T]; X_{t^-}=X_t =0\}$.
\end{Lemma}
\begin{proof}
Since in the case $p\in [1,2)$ the function $\theta$ is not smooth enough to apply It\^o's formula we use an approximation. Let $\eps > 0$ and let us consider the function $u_\eps(y) = (|y|^2 + \eps^2)^{1/2}$. It is a smooth function and we have
$$\frac{\partial u_\eps^p}{\partial y_i} (y)= p y_i u_\eps(y)^{p-2}, \frac{\partial^2 u_\eps^p}{\partial y_i \partial y_j} (y)= p u_\eps(y)^{p-2} \delta_{i,j} + p(p-2)y_iy_j u_\eps(y)^{p-4}.$$
We apply It\^o's formula to $X$:
\begin{eqnarray} \nonumber
&& u_\eps(X_t)^p = u_\eps(X_0)^p + \int_{0}^{t} p  u_\eps(X_s)^{p-2} X_s K_s ds  +  p  \int_{0}^{t}  u_\eps(X_s)^{p-2}  X_{s} Z_s dW_s    \\ \nonumber
& &\quad  +  p  \int_{0}^{t}  u_\eps(X_{s^-})^{p-2}  X_{s^-} d M_s +  p  \int_{0}^{t}  u_\eps(X_{s^-})^{p-2}  X_{s^-}   \int_\cU \psi_s(u) \tpi(du,ds) \\ \nonumber
& &\quad  + \frac{1}{2}\int_0^t \trace \left( D^2(u_\eps^p)(X_s) Z_sZ_s^* \right) ds \\ \nonumber
& &\quad + \int_{0}^{t}  \int_\cU \left( u_\eps(X_{s^-}+\psi_s(u))^p - u_\eps(X_{s^-})^p - pX_{s^-} u_\eps(X_{s^-})^{p-2} \psi_s(u) \right) \pi(du,ds) \\  \nonumber
&& \quad + \frac{1}{2} \int_0^t \sum_{1\leq i,j\leq d}  \frac{\partial^2 u_\eps^p}{\partial y_i \partial y_j} (X_s) d[ M^i, M^j ]^c_s \\ \label{eq:Ito_formula_approx}
&& \quad + \sum_{0< s \leq t}  \left( u_\eps(X_{s^-}+\Delta M_s)^p - u_\eps(X_{s^-})^p - pX_{s^-} u_\eps(X_{s^-})^{p-2} \Delta M_s \right).
\end{eqnarray}
Now we have to pass to the limit when $\eps$ goes to 0. As in \cite{bria:dely:hu:03} for the terms involving the first derivatives of $u_\eps$ we have
\begin{eqnarray*}
\int_{0}^{t}  u_\eps(X_s)^{p-2} X_s K_s ds & \longrightarrow & \int_0^t |X_s|^{p-1} \check{X}_s K_s ds \\
\int_{0}^{t}  u_\eps(X_s)^{p-2}  X_{s} Z_s dW_s & \longrightarrow & \int_0^t |X_s|^{p-1} \check{X}_s Z_s dW_s \\
\int_{0}^{t}  u_\eps(X_s)^{p-2}  X_{s}  \int_\cU \psi_s(u) \tpi(du,ds)  & \longrightarrow & \int_0^t |X_s|^{p-1} \check{X}_s  \int_\cU \psi_s(u) \tpi(du,ds)  \\
 \int_{0}^{t}  u_\eps(X_{s^-})^{p-2}  X_{s^-} d M_s & \longrightarrow &  \int_0^t |X_{s^-}|^{p-1} \check{X}_{s^-} dM_s.
\end{eqnarray*}
Moreover by the same arguments (convexity of $u_\eps$ and Fatou's lemma) the two following terms
$$ \int_{0}^{t}  \int_\cU \left[ u_\eps(X_{s^-}+\psi_s(u))^p - u_\eps(X_{s^-})^p - pX_{s^-} u_\eps(X_{s^-})^{p-2} \psi_s(u) \right] \pi(du,ds) $$
$$\sum_{0< s \leq t}  \left[ u_\eps(X_{s^-}+\Delta M_s)^p - u_\eps(X_{s^-})^p - pX_{s^-} u_\eps(X_{s^-})^{p-2} \Delta M_s \right]$$
converge, at least in probability, to
$$ \int_{0}^{t}  \int_\cU \left[ |X_{s^-}+\psi_s(u)|^p -|X_{s^-}|^p - p|X_{s^-}|^{p-1} \check{X}_{s^-}  \psi_s(u) \right] \pi(du,ds) $$
$$\sum_{0< s \leq t}  \left[ |X_{s^-}+\Delta M_s|^p - |X_{s^-}|^p - p|X_{s^-}|^{p-1}  \check{X}_{s^-} \Delta M_s \right].$$
Now for a non negative symmetric matrix $\Gamma \in \R^{d\times d}$
\begin{eqnarray} \nonumber
&& \sum_{1\leq i,j \leq d} D^2\theta(y)_{i,j} \Gamma_{i,j}   =  p u_\eps(y)^{p-2} \trace(\Gamma) + p(p-2) u_\eps(y)^{p-4} (y^*)\Gamma y \\ \nonumber
& & =  p(2-p) \left( \frac{|y|}{u_\eps(y)}\right)^{4-p} |y|^{p-2} \left[ \trace(\Gamma) - (\check{y})^* \Gamma \check{y} \right] \1_{y\neq 0} \\ \label{eq:meyer_ito_1}
& &\qquad + p(p-1) \left( \frac{|y|}{u_\eps(y)}\right)^{4-p} |y|^{p-2}\trace(\Gamma)\1_{y\neq 0} + p\eps^2 u_\eps(y)^{p-4} \trace(\Gamma) .
\end{eqnarray}
We have the following properties:
\begin{itemize}
\item $ \trace(\Gamma) \geq (\check{y})^* \Gamma \check{y}$,
\item $ \frac{|y|}{u_\eps(y)} \nearrow \1_{y\neq 0}$ as $\eps \searrow 0$.
\end{itemize}
For $\Gamma_s = Z_s Z^*_s$, by monotone convergence we obtain that
$$\int_0^t  \left( \frac{|X_{s}|}{u_\eps(X_{s})}\right)^{4-p} |X_{s}|^{p-2} \left\{(2-p)\left[ |Z_s|^2- (\check{X_{s}})^* Z_sZ_s^* \check{X_{s}} \right] + (p-1)|Z_s|^2 \right\} \1_{X_s\neq 0} ds$$
converges $\P$-a.s. for all $0\leq t \leq T$ to
$$\int_0^t  |X_{s}|^{p-2} \left\{(2-p)\left[ |Z_s|^2 - (\check{X_{s}})^* Z_s Z_s^* \check{X_{s}} \right] + (p-1)|Z_s|^2 \right\} \1_{X_s\neq 0} ds.$$
And for the integral w.r.t.\ the matrix $[ M, M ]^c = ([ M^i,M^j]^c_t, \ 1\leq i,j\leq d)$ we have the same result and the convergence to
$$\int_0^t  |X_{s}|^{p-2} \1_{X_s\neq 0}  \left\{(2-p)\left[ d[ M ]^c_s - (\check{X_{s}})^* d[M,M]^c_s \check{X_{s}} \right] + (p-1)d[ M ]^c_s \right\} ,$$
where $[M]^c = \sum_{i=1}^d [M^i, M^i]^c$. There is one remaining term in \eqref{eq:Ito_formula_approx}:
$$C_\eps^p(t)=p  \eps^2 \int_0^t u_\eps(X_s)^{p-4} \left[ |Z_s|^2 ds + d[ M]^c_s \right].$$
It follows from \eqref{eq:Ito_formula_approx} and the considerations above that this term converges to a process $L^p(t)$. By the same arguments as in \cite{bria:dely:hu:03}, we can prove that $L^p(t) = 0$ if $p > 1$. Indeed if $p\geq 4$, $u_\eps(X_s)^{p-4}$ converges in $L^1(\Omega\times (0,T))$ and if $1<p<4$, using H\"older inequality with $\theta = (4-p)/3 \in (0,1)$:
$$C_\eps^p(t)\leq p \left( \int_0^t \eps^2 u_\eps(X_s)^{-3}\left[ |Z_s|^2 ds+d[M]^c_s\right]  \right)^\theta \left( \int_0^t \eps^2 \left[ |Z_s|^2 ds+d[M]^c_s\right]  \right)^{1-\theta}.$$
Since the first term in the right-hand side converges to $L^1(t)$, $C_\eps^p(t)$ tends to zero.

Let us denote by $L(t)$ the process $L^1(t)$ and we proceed almost as in Chapter IV.7 (see Theorem 69) in \cite{prot:04}. By letting $\eps$ tend to zero in \eqref{eq:Ito_formula_approx} we obtain that $L$ satisfies \eqref{eq:Ito_formula_x_p}. By identifying the jumps on both sides of the equation it follows that $L$ is continuous. Moreover, $L$ is non decreasing in time. Now let us set $A=\{t\in[0;T]; \ X_{t^-}=X_t =0\}$. If $t$ is in the interior of $A$, then there exists $\delta > 0$ such that $X_s = 0$ whenever $|t-s|\leq \delta$ and the quadratic variation of $X$ is constant on the interval $[t - \delta; t + \delta]$ and then $Z_s = 0$ and $[M]_s=0$ almost everywhere on this interval. Hence $L$ does not increase in the interior of $A$.
Now assume that $t$ is in the interior of the complement of $A$. Since $L$ is continuous, the associated measure $dL$ is diffusive and does not charge any countable set. In particular, as $X$ is \cad, $dL$ does not charge the points where $X$ jumps. Hence, we can assume that $X_t=X_{t-}$. Then there exists some $\delta>0$ such that $X_s\neq 0$ for $|t-s|<\delta$. Consequently, $L(s) = L(t)$ for $|t-s|<\delta$, which completes the proof.
\end{proof}

As a byproduct to the proof we obtain the following lemma.
\begin{Lemma}\label{lem:Ito_formula_x_p_consequence}
For $p\in (1,2)$, for any $t \geq 0$ 
$$\int_0^t \1_{X_s=0} \left[ |Z_s|^2 ds+d[M]^c_s\right] = 0.$$
\end{Lemma}
\begin{proof}
Indeed $C_\eps^p$ can be written as follows:
\begin{eqnarray*}
C_\eps^p(t) & = & p  \eps^2 \int_0^t u_\eps(X_s)^{p-4} \left[ |Z_s|^2 ds + d[ M]^c_s \right] \\
& = & p  \eps^2 \int_0^t  \left( |X_s|^2 +\eps^2 \right)^{p/2-2} \1_{X_s \neq 0} \left[ |Z_s|^2 ds+ d[ M]^c_s \right] \\
& & \quad  + p  \eps^{p-2} \int_0^t  \1_{X_s = 0} \left[ |Z_s|^2 ds + d[ M]^c_s \right].
\end{eqnarray*}
Hence $C_\eps^p$ can converge to zero if and only if the last term is zero. 
\end{proof}

\begin{Coro} \label{coro:ito_form_xp}
If $(Y,Z,\psi,M)$ is a solution of BSDE \eqref{eq:gene_BSDE}, $p\in [1,2)$, $c(p) =\frac{p(p-1)}{2}$ and $0\leq t \leq r \leq T$, then:
\begin{eqnarray} \nonumber
&& |Y_t|^p \leq |Y_r|^p + p \int_t^r |Y_s|^{p-1} \check{Y}_s f(s,Y_s,Z_s,\psi_s) ds  -  p \int_t^r |Y_s|^{p-1} \check{Y}_s Z_s dW_s  \\ \nonumber
& &\quad  -  p  \int_t^r |Y_{s^-}|^{p-1} \check{Y}_{s^-} dM_s -  p \int_t^r |Y_s|^{p-1} \check{Y}_s  \int_\cU \psi_s(u) \tpi(du,ds)   \\ \nonumber
& &\quad -  \int_{t}^{r}  \int_\cU \left[ |Y_{s^-}+\psi_s(u)|^p -|Y_{s^-}|^p - p|Y_{s^-}|^{p-1} \check{Y}_{s^-}  \psi_s(u) \right] \pi(du,ds) \\  \nonumber
&& \quad - \sum_{0< t \leq r}  \left[ |Y_{s^-}+\Delta M_s|^p - |Y_{s^-}|^p - p|Y_{s^-}|^{p-1} \check{Y}_{s^-}  \Delta M_s \right] \\ \nonumber
&&\quad - c(p) \int_t^r  |Y_{s}|^{p-2} |Z_s|^2 \1_{Y_s\neq 0} ds -c(p) \int_t^r  |Y_{s}|^{p-2}  \1_{Y_s\neq 0} d[ M ]^c_s .
\end{eqnarray}
Moreover if $p\in (1,2)$, then 
$\int_0^t \1_{Y_s=0} \left[ |Z_s|^2 ds+d[M]^c_s\right] = 0$.
\end{Coro}
\begin{proof}
A direct consequence of Lemmas \ref{lem:Ito_formula_x_p} and \ref{lem:Ito_formula_x_p_consequence}.
\end{proof}

\begin{Lemma} \label{lem:control_jumps_part}
For $p\in [1,2)$, the non-decreasing process involving the jumps of $Y$ controls the quadratic variation as follows:
\begin{eqnarray*}
&&\sum_{0< s \leq t}  \left[ |Y_{s^-}+\Delta M_s|^p - |Y_{s^-}|^p - p|Y_{s^-}|^{p-1} \check{Y}_{s^-} \Delta M_s \right]  \\
&& \quad \geq c(p)  \sum_{0< s \leq t}|\Delta M_s|^2  \left( |Y_{s^-}|^2 \vee  |Y_{s^-} + \Delta M_s|^2 \right)^{p/2-1} \1_{|Y_{s^-}|\vee |Y_{s^-} + \Delta M_s| \neq 0}.
\end{eqnarray*}
The same holds for the jumps due to the Poisson random measure.
\end{Lemma}
\begin{proof}
We proceed as in the proof of Proposition \ref{prop:Lp_estimates} and we use the approximation of Lemma \ref{lem:Ito_formula_x_p}. Using Taylor expansion we obtain
\begin{eqnarray*}
&&\sum_{0< s \leq t}  \left[ u_\eps(Y_{s^-}+\Delta M_s)^p - u_\eps(Y_{s^-})^p - pY_{s^-} u_\eps(Y_{s^-})^{p-2} \Delta M_s \right] \\
&&\quad =  \sum_{0< s \leq t} \int_0^1 (1-a)  \Delta M_s D^2 (u_\eps(Y_{s^-} + a \Delta M_s)^p)  \Delta M_s da \\
&& \quad = p \sum_{0< s \leq t} \int_0^1 (1-a) | \Delta M_s|^2 u_\eps(Y_{s^-} + a \Delta M_s)^{p-2} da \\
&& \qquad + p(p-2) \sum_{0< s \leq t} \int_0^1 (1-a) \langle \Delta M_s , Y_{s^-} + a \Delta M_s \rangle^2 u_\eps(Y_{s^-} + a \Delta M_s)^{p-4} da \\
&& \quad \geq p(p-1)  \sum_{0< s \leq t}|\Delta M_s|^2  \int_0^1 (1-a) u_\eps(Y_{s^-} + a \Delta M_s)^{p-2} da.
\end{eqnarray*}

Since $|Y_{s^-} + a \Delta M_s| = |(1-a)Y_{s^-} + a (Y_{s^-} + \Delta M_s)| \leq |Y_{s^-}| \vee  |Y_{s^-} + \Delta M_s|$, we obtain:
 \begin{eqnarray*}
&&\sum_{0< s \leq t}  \left[ u_\eps(Y_{s^-}+\Delta M_s)^p - u_\eps(Y_{s^-})^p - pY_{s^-} u_\eps(Y_{s^-})^{p-2} \Delta M_s \right]  \\
&& \quad \geq \frac{p(p-1)}{2}  \sum_{0< s \leq t}|\Delta M_s|^2  \left( |Y_{s^-}|^2 \vee  |Y_{s^-} + \Delta M_s|^2 + \eps^2 \right)^{p/2-1} .
\end{eqnarray*}
Passing to the limit as $\eps$ goes to zero, we obtain:
 \begin{eqnarray*}
&&\sum_{0< s \leq t}  \left[ |Y_{s^-}+\Delta M_s|^p - |Y_{s^-}|^p - p|Y_{s^-}|^{p-1} \check{Y}_{s^-} \Delta M_s \right]  \\
&& \quad \geq \frac{p(p-1)}{2}  \sum_{0< s \leq t}|\Delta M_s|^2  \left( |Y_{s^-}|^2 \vee  |Y_{s^-} + \Delta M_s|^2 \right)^{p/2-1}  \1_{|Y_{s^-}|\vee |Y_{s^-} + \Delta M_s| \neq 0}.
\end{eqnarray*}
This achieves the proof ot the lemma. 
\end{proof}

\begin{Remark} \label{rem:p_geq_2}
If $p\geq 2$, then the conclusions of Corollary \ref{coro:ito_form_xp}
 and of Lemma \ref{lem:control_jumps_part} hold with $c(p)=p/2$. 
\end{Remark}

From now on, we assume that $p\in(1,2)$. The proof of the existence of a unique solution of BSDE \eqref{eq:gene_BSDE} in the space $\cE^p(0,T)$ is based on the following technical result. This estimates are also proved in \cite{klim:14}, Proposition 5.3, but in dimension 1. Moreover this estimate looks very similar to Inequality \eqref{eq:Lp_estim_sol_p_geq 2}. The main difference is that for $p< 2$, or $p/2< 1$, the compensator of a martingale does not control the predictable projection (see \cite{leng:lepi:prat:80} and the counterexample therein). We say that the condition (C) holds if $\P$-a.s.
$$\langle  \check{y}, f(t,y,z,\psi) \rangle \leq f_t + \al |y| + K |z| + K \|\psi\|_{L^2_\mu},$$
with $K \geq 0$ and $f_t$ is a non-negative progressively measurable process. Let us denote $F = \int_0^T f_r dr$.
\begin{Prop}\label{prop:apriori_lp}
Let the assumption (C) hold and let be $(Y,Z,\psi,M)$ be a solution of BSDE \eqref{eq:gene_BSDE} and assume moreover that $F^p$ is integrable and $Y \in \bD^p(0,T)$. Then $(Z,\psi,M)$ belongs to $\bH^p(0,T) \times L^p_\pi(0,T)  \times \bM^p(0,T)$ and there exists a constant $C$ depending on $p$, $K$ and $T$ such that 
\begin{eqnarray*}
&& \E \left[\sup_{t\in[0,T]}   |Y_t|^p +  \left( \int_0^T  |Z_t|^2 dt \right)^{p/2} + \left(  [ M ]_T \right)^{p/2}+  \left(  \int_0^T  \int_{\cU} |\psi_s(u)|^2 \pi(du, ds) \right)^{p/2}\right. \\
&& \qquad \left.  + \left( \int_0^T \int_\cU |\psi_s(u)|^2 \mu(du) ds \right)^{p/2} \right]\leq C \E \left[|\xi|^p  + \left( \int_0^T f_r dr \right)^p \right].
\end{eqnarray*}
\end{Prop}
Once again let us emphasize that the dependence of $f$ w.r.t. $\psi$ implies that we have to control the two expectations containing the term $\psi$. A crucial point in the proof of Proposition \ref{prop:Lp_estimates} was Inequality \eqref{eq:pred_quad_var}. Now in the case $p< 2$ we can not control (see \cite{leng:lepi:prat:80}, Section 4) the expectation of the predictable projection:
$$\E \left(  \int_0^T  \int_{\cU} |\psi_s(u)|^2 \pi(du, ds) \right)^{p/2}$$
with the expectation of the quadratic variation:
$$\E \left(  \int_0^T  \int_{\cU} |\psi_s(u)|^2 \pi(du, ds) \right)^{p/2}.$$

\begin{proof}
For some $a\in \R$, let us define $\wtil Y_t =e^{at}Y_t$, $\wtil Z_t =e^{at} Z_t$, $\wtil \psi_t = e^{at} \psi_t$ and $d\wtil M_t = e^{at} dM_t$. $(\wtil Y, \wtil Z, \wtil \psi, \wtil M)$ satisfies an analogous BSDE with terminal condition $\wtil \xi = e^{a T}\xi$ and generator
$$\wtil f(t,y,z,\psi) = e^{a t} f(t,e^{-a t} y, e^{-a t}z, e^{-a t} \psi) - a y.$$
$\wtil f$ satisfies assumptions ${\bf (H_{ex})}$ and (C) with $\wtil K = K$ and $\wtil \al = \al - a$. We choose $a$ large enough such that 
$$\wtil \al + 2K^2/(p-1) \leq 0.$$
Since we are working on a compact time interval, the integrability
conditions are equivalent with or without the superscript \textasciitilde. We omit the superscript \textasciitilde \ for notational convenience. 

\noindent {\bf Step 1:} We prove first that if $\al+2K^2/(p-1) \leq 0$, there exists a constant $\kappa_p$ such that 
$$\E (Y_*^p) \leq \kappa_p \E\left( X \right),$$
where 
$$Y_* = \sup_{t\in [0,T]} |Y_t|, \quad \mbox{and} \quad X =  |\xi|^p +   p \int_0^T |Y_s|^{p-1} f_s ds.$$
We apply Corollary \ref{coro:ito_form_xp} for $\tau \in \cT_T$:
\begin{eqnarray*}
&& |Y_{t\wedge \tau}|^p + c(p) \int_{t\wedge \tau}^{\tau}  |Y_{s}|^{p-2} |Z_s|^2 \1_{Y_s\neq 0} ds + c(p) \int_{t\wedge \tau}^{\tau}  |Y_{s}|^{p-2}  \1_{Y_s\neq 0} d[ M ]^c_s \\
&& \leq |Y_{\tau}|^p+ p \int_{t\wedge \tau}^{\tau} |Y_s|^{p-1} \check{Y}_s f(s,Y_s,Z_s,\psi_s) ds  -  p \int_{t\wedge \tau}^{\tau} |Y_s|^{p-1} \check{Y}_s Z_s dW_s  \\
& &\quad  -  p  \int_{t\wedge \tau}^{\tau} |Y_{s^-}|^{p-1} \check{Y}_{s^-} dM_s -  p \int_{t\wedge \tau}^{\tau} |Y_s|^{p-1} \check{Y}_s  \int_\cU \psi_s(u) \tpi(du,ds)  \\
& &\quad -  \int_{t\wedge \tau}^{\tau}  \int_\cU \left[ |Y_{s^-}+\psi_s(u)|^p -|Y_{s^-}|^p - p|Y_{s^-}|^{p-1} \check{Y}_{s^-} \psi_s(u) \right] \pi(du,ds) \\
&& \quad - \sum_{t\wedge \tau< s \leq \tau}  \left[ |Y_{s^-}+\Delta M_s|^p - |Y_{s^-}|^p - p|Y_{s^-}|^{p-1}  \check{Y}_{s^-} \Delta M_s \right] .
\end{eqnarray*}
With the assumption on $f$ this becomes
\begin{eqnarray*}
&& |Y_{t\wedge \tau}|^p + c(p) \int_{t\wedge \tau}^{\tau}  |Y_{s}|^{p-2} |Z_s|^2 \1_{Y_s\neq 0} ds + c(p) \int_{t\wedge \tau}^{\tau}  |Y_{s}|^{p-2}  \1_{Y_s\neq 0} d[ M ]^c_s \\
&& \leq |Y_{\tau}|^p + p \int_{t\wedge \tau}^{\tau} \left( |Y_s|^{p-1} f_s + \al |Y_s|^p  \right) ds + pK \int_{t\wedge \tau}^{\tau} |Y_s|^{p-1} |Z_s| ds  \\
&& \qquad + pK \int_{t\wedge \tau}^{\tau} |Y_s|^{p-1} \|\psi_s\|_{L^2_\mu} ds  -  p \int_{t\wedge \tau}^{\tau} |Y_s|^{p-1} \check{Y}_s Z_s dW_s  \\
& &\qquad  -  p  \int_{t\wedge \tau}^{\tau} |Y_{s^-}|^{p-1} \check{Y}_{s^-} dM_s -  p \int_{t\wedge \tau}^{\tau} |Y_s|^{p-1} \check{Y}_s  \int_\cU \psi_s(u) \tpi(du,ds)  \\
& &\quad -  \int_{t\wedge \tau}^{\tau}  \int_\cU \left[ |Y_{s^-}+\psi_s(u)|^p -|Y_{s^-}|^p - p|Y_{s^-}|^{p-1}  \check{Y}_{s^-} \psi_s(u) \right] \pi(du,ds) \\
&& \quad - \sum_{t\wedge \tau< s \leq \tau}  \left[ |Y_{s^-}+\Delta M_s|^p - |Y_{s^-}|^p - p|Y_{s^-}|^{p-1} \check{Y}_{s^-}  \Delta M_s \right]
\end{eqnarray*}
Moreover
 \begin{eqnarray*}
pK |Y_s|^{p-1} |Z_s| & \leq & \frac{pK^2}{p-1} |Y_s|^p + \frac{c(p)}{2}  |Y_{s}|^{p-2} |Z_s|^2 \1_{Y_s\neq 0} \\
pK |Y_s|^{p-1} \|\psi_s\|_{L^2_\mu} & \leq & \frac{pK^2}{p-1} |Y_s|^p + \frac{c(p)}{2}  |Y_{s}|^{p-2} \|\psi_s\|^2_{L^2_\mu} \1_{Y_s\neq 0} 
\end{eqnarray*}
and from the previous lemma
 \begin{eqnarray*}
&& \int_{t\wedge \tau}^{\tau}  \int_\cU \left[ |Y_{s^-}+\psi_s(u)|^p -|Y_{s^-}|^p - p|Y_{s^-}|^{p-1}  \check{Y}_{s^-} \psi_s(u) \right] \pi(du,ds) \\
&& \quad \geq c(p) \int_{t\wedge \tau}^{\tau}  \int_\cU |\psi_s(u)|^2  \left( |Y_{s^-}|^2 \vee  |Y_{s^-} +\psi_s(u)|^2 \right)^{p/2-1} \1_{|Y_{s^-}|\vee |Y_{s^-} + \psi_s(u)| \neq 0} \pi(du,ds).
\end{eqnarray*}
and
 \begin{eqnarray*}
&&\sum_{t\wedge \tau< s \leq \tau}  \left[ |Y_{s^-}+\Delta M_s|^p - |Y_{s^-}|^p - p|Y_{s^-}|^{p-1} \check{Y}_{s^-} \Delta M_s \right]  \\
&& \quad \geq c(p)  \sum_{t\wedge \tau< s \leq \tau}|\Delta M_s|^2  \left( |Y_{s^-}|^2 \vee  |Y_{s^-} + \Delta M_s|^2 \right)^{p/2-1}  \1_{|Y_{s^-}|\vee |Y_{s^-} + \Delta M_s| \neq 0}
\end{eqnarray*}
Therefore we deduce the following inequality:
\begin{eqnarray} \nonumber
&& |Y_{t\wedge \tau}|^p + \frac{c(p)}{2} \int_{t\wedge \tau}^\tau  |Y_{s}|^{p-2} |Z_s|^2 \1_{Y_s\neq 0} ds + c(p) \int_{t\wedge \tau}^\tau  |Y_{s}|^{p-2}  \1_{Y_s\neq 0} d[ M ]^c_s \\  \nonumber
&& + c(p) \sum_{t\wedge \tau< s \leq \tau}  \left( |Y_{s^-}|^2 \vee  |Y_{s^-} + \Delta M_s|^2 \right)^{p/2-1}  \1_{|Y_{s^-}|\vee |Y_{s^-} + \Delta M_s| \neq 0} |\Delta M_s|^2\\  \nonumber
&& + c(p)  \int_{t\wedge \tau}^{\tau}  \int_\cU |\psi_s(u)|^2  \left( |Y_{s^-}|^2 \vee  |Y_{s^-} +\psi_s(u)|^2 \right)^{p/2-1} \1_{|Y_{s^-}|\vee |Y_{s^-} + \psi_s(u)| \neq 0}\pi(du,ds) \\  \nonumber
&& - \frac{c(p)}{2}  \int_{t\wedge \tau}^{\tau}  |Y_{s}|^{p-2} \|\psi_s\|^2_{L^2_\mu} \1_{Y_s\neq 0} ds \\  \nonumber
&& \leq |Y_\tau|^p +   p \int_{t\wedge \tau}^\tau \left( |Y_s|^{p-1} f_s + \al |Y_s|^p  \right) ds  +p \int_{t\wedge \tau}^\tau \frac{2K^2}{p-1} |Y_s|^p ds  \\ \label{eq:Lp_apriori_estim_1}
&& \quad  -  p \int_{t\wedge \tau}^\tau |Y_s|^{p-1} \check{Y}_s \left( Z_s dW_s  + dM_s +  \int_\cU \psi_s(u) \tpi(du,ds) \right).
\end{eqnarray}
At the very beginning of this proof we suppose that $\al + \frac{2K^2}{p-1}\leq 0 $. Thus the term $(\alpha + \frac{2K^2}{p-1} )  \int_{t\wedge \tau}^\tau |Y_s|^p ds$ disappears. Let us define $\tau_k$ as a fundamental sequence of stopping times for the local martingale 
$$\int_{0}^{.} |Y_s|^{p-1} \check{Y}_s \left( Z_s dW_s  + dM_s +  \int_\cU \psi_s(u) \tpi(du,ds) \right).$$
Let 
$$\hat \tau_k  = \inf \left\{ t\geq 0, \quad  \int_{0}^{t}  \int_\cU |\psi_s(u)|^2  \left( |Y_{s^-}|^2 \vee  |Y_{s}|^2 \right)^{p/2-1} \1_{|Y_{s^-}|\vee |Y_{s}| \neq 0} \pi(du, ds) \geq k \right\}\wedge T.$$
We take $\tau = \tau_k \wedge \hat \tau_k$. Now we have:
\begin{eqnarray}\nonumber
&& \E  \int_{0}^{\tau}  \int_\cU |\psi_s(u)|^2  \left( |Y_{s^-}|^2 \vee  |Y_{s^-} +\psi_s(u)|^2 \right)^{p/2-1} \1_{|Y_{s^-}|\vee |Y_{s^-} + \psi_s(u)| \neq 0} \pi(du,ds) \\ \nonumber
&& \quad =  \E  \int_{0}^{\tau}  \int_\cU |\psi_s(u)|^2  \left( |Y_{s^-}|^2 \vee  |Y_{s}|^2 \right)^{p/2-1} \1_{|Y_{s^-}|\vee |Y_{s}| \neq 0} \pi(du, ds) \\ \label{eq:local_mart_jump}
&& \quad = \E  \int_{0}^{\tau}  \int_\cU |\psi_s(u)|^2  |Y_{s}|^{p-2} \1_{Y_{s} \neq 0} \mu(du) ds
\end{eqnarray}
the last equality coming from the localization due to $\hat \tau_k$ and since the set $\{ s \geq 0, \ Y_{s} \neq Y_{s^-} \}$ is countable. Taking the expectation in \eqref{eq:Lp_apriori_estim_1}, we get
\begin{eqnarray} \nonumber
&&\E (|Y_{t\wedge \tau}|^p) + \frac{c(p)}{2} \E \int_0^\tau  |Y_{s}|^{p-2} |Z_s|^2 \1_{Y_s\neq 0} ds + c(p) \E \int_0^\tau  |Y_{s}|^{p-2}  \1_{Y_s\neq 0} d[ M ]^c_s \\ \nonumber
&& \qquad + c(p) \E \sum_{0< s \leq \tau}  \left( |Y_{s^-}|^2 \vee  |Y_{s^-} + \Delta M_s|^2 \right)^{p/2-1}  \1_{|Y_{s^-}|\vee |Y_{s^-} + \Delta M_s| \neq 0}|\Delta M_s|^2\\ \nonumber
&& \qquad + \frac{c(p)}{2} \E  \int_{t}^{\tau}  \int_\cU |\psi_s(u)|^2  \left( |Y_{s^-}|^2 \vee  |Y_{s^-} +\psi_s(u)|^2 \right)^{p/2-1} \1_{|Y_{s^-}|\vee |Y_{s^-} + \psi_s(u)| \neq 0} \pi(du,ds) \\ \label{eq:Lp_estim_p_leq_2}
 && \quad \leq  \E \left( |Y_\tau|^p +   p \int_{0}^\tau  |Y_s|^{p-1} f_s ds \right).
 \end{eqnarray}
We also obtain:
$$ \frac{c(p)}{2}\E  \int_{0}^\tau  \int_\cU |\psi_s(u)|^2  |Y_{s}|^{p-2} \1_{Y_{s} \neq 0} \mu(du) ds \leq  \E \left( |Y_\tau|^p +   p \int_{0}^\tau  |Y_s|^{p-1} f_s ds \right).$$ 
Recall that $X$ is the quantity
$$X =  |\xi|^p +   p \int_0^T |Y_s|^{p-1} f_s ds.$$
Then we can pass to the limit on $k$ in \eqref{eq:Lp_estim_p_leq_2}, and we obtain the same estimate for $\tau=T$ and $\E(X)$ on the right-hand side. 
Let us split the local martingale of \eqref{eq:Lp_apriori_estim_1} into three terms:
\begin{eqnarray*}
\Gamma_t & = & \int_0^t |Y_s|^{p-1} \check{Y}_s Z_s dW_s ,\\
\Theta_t & = & \int_0^t |Y_s|^{p-1} \check{Y}_s dM_s, \quad \Xi_t =  \int_0^t |Y_s|^{p-1} \check{Y}_s  \int_\cU \psi_s(u) \tpi(du,ds)  .
\end{eqnarray*}
Then using \eqref{eq:Lp_apriori_estim_1} and BDG inequality
\begin{eqnarray*}
\E (Y_*^p) & \leq & \E\left( X \right) + k_p \E \left( [\Gamma]_T^{1/2} + [\Theta]_T^{1/2} + [\Xi]_T^{1/2} \right).
\end{eqnarray*}
The bracket $[\Gamma]_T^{1/2}$ can be handled as in \cite{bria:dely:hu:03}:
\begin{eqnarray*}
k_p \E \left(  [\Gamma]_T^{1/2} \right) & \leq & \frac{1}{6}  \E \left( Y_*^{p} \right) +  \frac{3k_p^2 }{2} \E \left( \int_0^T  |Y_{s}|^{p-2} |Z_s|^2 \1_{Y_s\neq 0} ds \right).
\end{eqnarray*}
For the other terms we have
\begin{eqnarray*}
k_p \E \left(  [\Xi]_T^{1/2} \right) & \leq & k_p \E \left( Y_*^{p/2} \left( \int_0^T  |Y_{s}|^{p-2} |\psi_s|^2 \1_{Y_s\neq 0} \pi(du,ds) \right)^{1/2} \right) \\
& \leq  & \frac{1}{6}  \E \left( Y_*^{p} \right) +  \frac{3k_p^2 }{2} \E \left( \int_0^T  |Y_{s}|^{p-2} \|\psi_s\|^2_{L^2_\mu} \1_{Y_s\neq 0} ds \right) ,
\end{eqnarray*}
and for $[\Theta]$ since $p>1$
\begin{eqnarray*}
k_p \E \left(  [\Theta]_T^{1/2} \right) & \leq & k_p \E\left[ \left( \int_0^T \left( |Y_{s^-}|^2 \vee  |Y_{s^-} + \Delta M_s|^2 \right)^{p-1}  \1_{|Y_{s^-}|\vee |Y_{s^-} + \Delta M_s| \neq 0} d[M]_s \right)^{1/2} \right] \\
& \leq & k_p \E\left[ \left( \sup_{s \in [0,T]} \left( |Y_{s^-}|^2 \vee  |Y_{s^-} + \Delta M_s|^2 \right)^{p/2}\right)^{1/2} \right. \\
&& \qquad \qquad \left. \left( \int_0^T \left( |Y_{s^-}|^2 \vee  |Y_{s^-} + \Delta M_s|^2 \right)^{p/2-1}  \1_{|Y_{s^-}|\vee |Y_{s^-} + \Delta M_s| \neq 0}d[M]_s \right)^{1/2} \right] \\
& \leq  & \frac{1}{6}  \E \left( Y_*^{p} \right) +  \frac{3k_p^2 }{2} \E \left( \int_0^T  |Y_{s^-}|^{p-2}  \1_{|Y_{s^-}| \neq 0} d[M]^c_s \right. \\
&& \qquad \qquad \left. + \sum_{0< s \leq T}  \left( |Y_{s^-}|^2 \vee  |Y_{s^-} + \Delta M_s|^2 \right)^{p/2-1} \1_{|Y_{s^-}|\vee |Y_{s^-} + \Delta M_s| \neq 0} |\Delta M_s|^2 \right).
\end{eqnarray*}
We deduce that there exists a constant depending only on $p$ such that
\begin{eqnarray*}
\E (Y_*^p) & \leq & \kappa_p \E\left( X \right).
\end{eqnarray*}

\noindent {\bf Step 2:} Let us derive now a priori estimates for the martingale part of the BSDE. We use Corollary \ref{coro:ito_form_xp}: 
\begin{eqnarray} \nonumber
&& \E \left( \int_0^T |Z_s|^2 ds \right)^{p/2}  =  \E \left( \int_0^T\1_{Y_s \neq 0} |Z_s|^2 ds\right)^{p/2} \\ \nonumber
&&\quad = \E \left( \int_0^T \left( Y_{s} \right)^{2-p}\left( Y_{s} \right)^{p-2}\1_{Y_s \neq 0}  |Z_s|^2ds \right)^{p/2} \\ \nonumber
&& \quad \leq \E \left[ \left( Y_{*}\right)^{p(2-p)/2} \left(\int_0^T \left( Y_{s} \right)^{p-2}\1_{Y_s \neq 0}  |Z_s|^2 ds \right)^{p/2}\right]\\ \nonumber
&& \quad \leq \left\{\E \left[ \left( Y_{*} \right)^{p}\right]\right\}^{(2-p)/2}  \left\{\E \int_0^T \left( Y_{s}\right)^{p-2} \1_{Y_s \neq 0} |Z_s|^2 ds \right\}^{p/2} \\ \label{eq:trick_control_mart_part}
&& \quad  \leq \frac{2-p}{2} \E \left[ \left( Y_{*} \right)^{p}\right] + \frac{p}{2} \E \int_0^T \left( Y_{s} \right)^{p-2} \1_{Y_s \neq 0} |Z_s|^2 ds
\end{eqnarray}
where we have used H\"older's and Young's inequality with $ \frac{2-p}{2} + \frac{p}{2}=1$. With Inequality \eqref{eq:Lp_estim_p_leq_2} we deduce:
\begin{eqnarray*}
&& \E \left( \int_0^T |Z_s|^2 ds \right)^{p/2}  \leq \frac{2-p}{2}  \E \left( |Y_{*}|^{p} \right) + \frac{p}{2}\E \int_0^T |Y_{s}|^{p-2} \1_{Y_s\neq 0} |Z_s|^2 ds \leq \wtil \kappa_p \E(X).
\end{eqnarray*}
The same argument can be used to control $[M]^c$. For the pure-jump part of $[M]$ we have using the function $u_\eps$ defined in the proof of Lemma \ref{lem:Ito_formula_x_p}:
\begin{eqnarray*}
&& \E \left( \sum_{0<s\leq T} |\Delta M_s|^2 \right)^{p/2} \\
&&\quad = \E \left( \sum_{0<s\leq T} \left( u_\eps(|Y_{s^-}| \vee  |Y_{s^-} + \Delta M_s|) \right)^{2-p}\left(u_\eps( |Y_{s^-}| \vee  |Y_{s^-} + \Delta M_s| )\right)^{p-2} |\Delta M_s|^2 \right)^{p/2} \\
&& \quad \leq \E \left[ \left( u_\eps(Y_{*})\right)^{p(2-p)/2} \left(\sum_{0<s\leq T} \left( u_\eps(|Y_{s^-}| \vee  |Y_{s^-} + \Delta M_s|) \right)^{p-2} |\Delta M_s|^2 \right)^{p/2}\right]\\
&& \quad \leq \left\{\E \left[ \left( u_\eps(Y_{*}) \right)^{p}\right]\right\}^{(2-p)/2} \\
&&\qquad \qquad \times \left\{\E\left(  \sum_{0< s \leq T} \left( u_\eps(|Y_{s^-}| \vee  |Y_{s^-} + \Delta M_s|) \right)^{p-2} |\Delta M_s|^2 \right) \right\}^{p/2} \\
&& \quad \leq \frac{2-p}{2} \E \left[ \left( u_\eps(Y_{*}) \right)^{p}\right] + \frac{p}{2}\E\left(  \sum_{0< s \leq T} \left( u_\eps(|Y_{s^-}| \vee  |Y_{s^-} + \Delta M_s| )\right)^{p-2} |\Delta M_s|^2 \right)
\end{eqnarray*}
Let $\eps$ go to zero. We use a convergence result, which is a direct consequence of the proof of Lemma \ref{lem:control_jumps_part}:
\begin{eqnarray*}
&& \E \left(  \sum_{0<s\leq T} |\Delta M_s|^2  \right)^{p/2} \\
&& \quad \leq \frac{2-p}{2} \E \left( |Y_{*}|^{p} \right) + \frac{p}{2}\E\left( \sum_{0\leq s < T} \left( |Y_{s^-}| \vee  |Y_{s^-} + \Delta M_s| \right)^{p-2}  \1_{|Y_{s^-}|\vee |Y_{s^-} + \Delta M_s| \neq 0}|\Delta M_s|^2 \right)\\
&& \quad \leq \wtil \kappa_p \E(X).
\end{eqnarray*}
The same argument shows that 
\begin{eqnarray*}
\E \left( \int_0^T \int_\cU |\psi_s(u)|^2 \pi(du,ds) \right)^{p/2} & \leq & \wtil \kappa_p \E(X). 
\end{eqnarray*}
For the estimate of 
$$\E \left( \int_0^T \int_\cU |\psi_s(u)|^2 \mu(du) ds \right)^{p/2} $$
we follow the same scheme with a localization argument and Equality \eqref{eq:local_mart_jump} to obtain that 
$$\lim_{\eps \to 0} \E \int_{0}^T\int_\cU  \left( u_\eps( |Y_{s^-}|\right)^{p-2}  |\psi_s(u)|^2\mu(du) ds =  \E \int_{0}^T\int_\cU   |Y_{s^-}|^{p-2} \1_{Y_{s^-}\neq 0} |\psi_s(u)|^2\mu(du) ds.$$

\noindent {\bf Step 3:} Now we prove the wanted estimate. Recall that we have found a constant $\hat \kappa_p$ such that 
 \begin{eqnarray*}
 &&\E\left[  |Y_*|^p +  \left( \int_0^T Z_s^2 ds \right)^{p/2} + [M]_T^{p/2} \right.\\
  &&\qquad \left. + \left( \int_0^T \int_\cU |\psi_s(u)|^2 \mu(du) ds \right)^{p/2} + \left( \int_0^T \int_\cU |\psi_s(u)|^2 \pi(du,ds) \right)^{p/2}\right] \leq \hat \kappa_p \E(X)
  \end{eqnarray*}
where 
$$X =  |\xi|^p +   p \int_0^T |Y_s|^{p-1} f_s ds.$$
But Young's inequality leads to
$$p \hat \kappa_p \int_0^T |Y_s|^{p-1} f_s ds\leq p \hat \kappa_p |Y_*|^{p-1} \int_0^T f_s ds \leq \frac{1}{2} \E \left( |Y_*|^p\right) + d_p \left(  \int_0^T f_s ds\right)^p.$$
Therefore we have proved that for $a$ such that $\al + \frac{3K^2}{(p-1)} \leq a$, then 
\begin{eqnarray*}
&&\E\left[ \left(  \sup_{t\in [0,T]} e^{at} |Y_t|^p \right)+  \left( \int_0^T e^{2as} Z_s^2 ds \right)^{p/2}  + \left( \int_0^T \int_\cU e^{2as} |\psi_s(u)|^2 \pi(du,ds) \right)^{p/2}\right. \\
&& \qquad \left.+ \left( \int_0^T e^{2as} d[M]_s\right)^{p/2} \right] \leq C \E \left[ e^{aT} |\xi|^p+ \left( \int_0^T e^{ar} f_r dr \right)^p\right]
\end{eqnarray*}
where $C$ just depends on $p$. This gives the desired estimate. 
\end{proof}

\begin{Thm} \label{thm:exist_uniq_sol_p_leq_2}
Under Assumptions ${\bf (H_{ex})}$ and (H5), there exists a unique solution $(Y,Z,\psi,M)$ in $\cE^p(0,T)$ to the BSDE \eqref{eq:gene_BSDE}. Moreover for some constant $C=C_{p,K,T}$
\begin{eqnarray*}
&& \E \left[\sup_{t\in[0,T]}   |Y_t|^p +  \left( \int_0^T  |Z_t|^2 dt \right)^{p/2} +  \left(  \int_0^T  \int_{\cU} |\psi_s(u)|^2 \pi(du, ds) \right)^{p/2} + \left(  [ M ]_T \right)^{p/2}\right] \\
&& \qquad \qquad \leq C \E \left[|\xi|^p  + \left( \int_0^T |f(r,0,0,0)| dr \right)^p \right].
\end{eqnarray*}
\end{Thm}
\begin{proof}
As for Theorem \ref{thm:exist_uniq_sol_gene_BSDE}, we follow the proof of the second step of Theorem 4.2 in \cite{bria:dely:hu:03}. We truncate $\xi$ and $f(t,0,0,0)$ to obtain $\xi_n$ and $f_n$ with $\|\xi_n\|_\infty \leq n$ and $|f_n(t,0,0,0)|\leq n$:
$$\xi_n = q_n(\xi), \qquad f_n(t,y,z,\psi) = f(t,y,z,\psi) - f(t,0,0,0) + q_n(f(t,0,0,0)),$$
with $q_n(x) = xn/(|x|\vee n)$. 
Thanks to Theorem \ref{thm:exist_uniq_sol_gene_BSDE}, we have a unique solution $(Y^n,Z^n,\psi^n,M^n)$ in $\cE^2$, and thus in $\cE^p$ for any $p> 1$. Now for any $m$ and $n$:
\begin{eqnarray*}
&& f_{m}(t,Y^{m}_t,Z^{m}_t,\psi^{m}_t) - f_{n}(t,Y^{n}_t,Z^{n}_t,\psi^{n}_t) = f(t,Y^{m}_t,Z^{m}_t,\psi^{m}_t) - f(t,Y^{n}_t,Z^{m}_t,\psi^{n}_t) \\
&&  \quad +  f(t,Y^{n}_t,Z^{m}_t,\psi^{m}_t) - f(t,Y^{n}_t,Z^{n}_t,\psi^{m}_t) + f(t,Y^{n}_t,Z^{n}_t,\psi^{m}_t) - f(t,Y^{n}_t,Z^{n}_t,\psi^{n}_t)\\
& &\quad + q_m(f(t,0,0,0)) - q_n(f(t,0,0,0))
\end{eqnarray*}
Hence
\begin{eqnarray*}
&& \langle \frac{(Y^m_t-Y^n_t)}{|Y^m_t-Y^n_t|} \1_{Y^m_t-Y^n_t\neq 0},f_{m}(t,Y^{m}_t,Z^{m}_t,\psi^{m}_t) - f_{n}(t,Y^{n}_t,Z^{n}_t,\psi^{n}_t) \rangle \\
&&\quad  \leq |q_m(f(t,0,0,0)) - q_n(f(t,0,0,0))|+ K|Z^m_t-Z^n_t| + K \|\psi^m_t-\psi^n_t\|_{L^2_\mu} .
\end{eqnarray*}
This inequality is Assumption (C) in Proposition \ref{prop:apriori_lp} with $\al=0$. This proposition shows that
\begin{eqnarray*}
&&\E \left[\sup_{t\in[0,T]}  |Y^{m}_t-Y^n_t|^p + \left( \int_0^T |Z^{m}_s-Z^n_s|^2 ds \right)^{p/2}+ \left( [ M^{m}-M^n  ]_T \right)^{p/2} \right. \\
&&\quad  \left.+  \left(  \int_0^T  \int_{\cU} |\psi^{m}_s(u)-\psi^n_s(u)|^2 \pi(du, ds) \right)^{p/2}+  \left(  \int_0^T  \int_{\cU} |\psi^{m}_s(u)-\psi^n_s(u)|^2 \mu(du) ds \right)^{p/2} \right] \\
&& \qquad \qquad \leq C \E \left[  |\xi_{m}-\xi_n|^p + \left( \int_0^T |q_{m}(f(r,0,0,0))-q_{n}(f(r,0,0,0))| dr \right)^p \right].
\end{eqnarray*}
Thus $(Y^n,Z^n,\psi^n,M^n)$ is a Cauchy sequence in $\cE^p$ and the conclusion follows.
\end{proof}

\section{Comparison Principle}

In this section we give some results which are derived from the previous sections. In the first part we assume that $d=1$ and aim at comparing two solutions $Y^1$ and $Y^2$ of the BSDE \eqref{eq:gene_BSDE} with coefficients $(\xi^1,f^1)$ and $(\xi^2,f^2)$. As in the papers of Barles et al. \cite{barl:buck:pard:97}, Royer \cite{royer:06}, Situ \cite{situ:05} or Quenez \& Sulem \cite{quen:sule:13}, we have to restrict the dependence of $f$ w.r.t.\ $\psi$. Some monotonicity w.r.t.\ $\psi$ is necessary. The following set of conditions will be denoted by ${\bf (H_{comp})}$. The three conditions (H1) to (H3) hold but assumption (H3) is replaced by:
\begin{enumerate}
\item[(H3')] $f$ is Lipschitz continuous w.r.t.\ $z$ with constant $K$ and for each $(y,z,\psi,\phi) \in \R\times \R^k \times (L^2_\mu)^2$, there exists a predictable process $\kappa = \kappa^{y,z,\psi,\phi} : \Omega \times [0,T] \times \cU \to \R$ such that:
$$f(t,y,z,\psi) - f(t,y,z,\phi) \leq  \int_\cU (\psi(u)-\phi(u)) \kappa^{y,z,\psi,\phi}_t(u) \mu(du)$$
with $\P\otimes m \otimes \mu$-a.e.  for any $(y,z,\psi,\psi')$,
\begin{itemize}
\item $-1 \leq \kappa^{y,z,\psi,\phi}_t(u)$
\item $|\kappa^{y,z,\psi,\phi}_t(u)|\leq \vartheta(u)$ where $\vartheta \in L^2_\mu$.
\end{itemize}
\end{enumerate}
Note that ${\bf (H_{comp})}$ implies ${\bf (H_{ex})}$. Indeed if (H3') is true we also have:
$$f(t,y,z,\psi) - f(t,y,z,\phi) \geq  \int_\cU (\psi(u)-\phi(u)) \kappa^{y,z,\phi,\psi}_t(u) \mu(du)$$
by changing the role of $\psi$ and $\phi$ in $\kappa$ and thus
$$|f(t,y,z,\psi) - f(t,y,z,\phi) | \leq \|\vartheta\|_{L^2_\mu}  \|\psi-\phi\|_{L^2_\mu}.$$

We follow the line of argument of \cite{quen:sule:13}. In particular we consider the Dol\'eans-Dade exponential local martingale: Let $\alpha$, $\beta$ be predictable processes integrable w.r.t.\ $dt$ and $dW_t$, respectively. Let $\gamma$ be a predictable process defined on $[0, T]\times \Omega \times \R$ integrable w.r.t.\ $\tilde \pi(du,ds)$. For any $0\leq t \leq s \leq T$ let $E$ be the solution of
$$dE_{t,s} = E_{t,s^-} \left[ \beta_s dW_s + \int_\cU \gamma_s(u)\tpi(du,ds)  \right],\qquad E_{t,t}=1,$$
and let $\Gamma$ be the solution of
\begin{equation} \label{eq:doleans_dade_exp}
d\Gamma_{t,s} = \Gamma_{t,s^-} \left[ \alpha_s ds + \beta_s dW_s + \int_\cU \gamma_s(u)\tpi(du,ds)  \right],\qquad \Gamma_{t,t}=1.
\end{equation}
Of course $\Gamma_{t,s} = \exp \left( \int_t^s \alpha_r dr \right) E_{t,s}$ and
$$E_{t,s} = \exp\left( \int_t^s \beta_r dW_r  - \frac{1}{2} \int_t^s \beta^2_r dr \right) \prod_{t<r\leq s} (1+ \gamma_r(\Delta X_r)) e^{-\gamma_r(\Delta X_r)}$$
with $X_t = \int_0^t \int_\cU u \pi(du,ds)$.
\begin{Lemma}
Assume that the processes $|\beta|$ and $\|\gamma\|_{L^2_\mu}$ are bounded and that $\al$ is bounded from above. Let $(Y,Z,\psi,M)$ be the solution of the following linear BSDE:
\begin{eqnarray} \nonumber
Y_t &=& \xi + \int_t^T \left[ f_s + \alpha_s Y_s + \beta_s Z_s + \int_\cU \gamma_s(u) \psi_s(u) \mu(du) \right] ds \\ \label{eq:linear_BSDE}
& - & \int_t^T\int_\cU \psi_s(u) \tpi(du,ds) -\int_t^T Z_sdW_s- \int_t^T dM_s.
\end{eqnarray}
Then $\Gamma$ is $q$-integrable for any $q\geq 2$, and the solution $(Y,Z,\psi,M)$ belongs to $\cE^p(0,T)$ if
$$\E\left( |\xi|^p + \int_0^T |f_s|^p ds \right) < +\infty.$$
Moreover
$$Y_t = \E \left[ \Gamma_{t,T} \xi + \int_t^T \Gamma_{t,s} f_s ds \bigg| \F_t \right].$$
\end{Lemma}
\begin{proof}
The integrability of $\Gamma_{t,s}$ is given in \cite{quen:sule:13}, Proposition A.1 and by Doob's inequality:
$$\E \sup_{s\in [t,T]} |\Gamma_{t,s}|^q \leq C_q \sup_{s\in [t,T]} \E |\Gamma_{t,s}|^q\leq C_q \E |\Gamma_{t,T}|^q < +\infty.$$
We follow the arguments of the proof of Theorem 3.4 in \cite{quen:sule:13}. Let $(Y,Z,\psi,M)$ be a solution. For $0\leq t \leq s \leq T$ set
$$\phi_s = Y_s\Gamma_{t,s} + \int_t^s  \Gamma_{t,r} f_r dr.$$
Then by integration by parts we obtain
\begin{eqnarray*}
d\phi_s  & = & \Gamma_{t,s^-} dY_s + Y_{s^-} d\Gamma_{t,s} + d [ \Gamma_{t,.} , Y ]_s +\Gamma_{t,s} f_s ds \\
& = &  \Gamma_{t,s^-}  \left( -f_s - \alpha_s Y_s - \beta_s Z_s - \int_\cU \gamma_s(u) \psi_s(u) \mu(du) \right) ds \\
&& \quad + \Gamma_{t,s^-} \int_\cU \psi_s(u) \tpi(du,ds)  + \Gamma_{t,s^-} Z_s dW_s+ \Gamma_{t,s^-}  dM_s. \\
& + &  Y_{s^-} \Gamma_{t,s^-}  \left( \alpha_s ds + \beta_s dW_s + \int_\cU \gamma_s(u)\tpi(du,ds)  \right) \\
& + & \Gamma_{t,s^-} \beta_s Z_s ds + \Gamma_{t,s^-} \int_\cU \psi_s(u)\gamma_s(u) \pi(du,ds)+\Gamma_{t,s} f_s ds \\
& = &  \Gamma_{t,s^-} \int_\cU( \psi_s(u) + Y_{s^-} \gamma_s(u)+\psi_s(u)\gamma_s(u)) \tpi(du,ds)  \\
& + & \Gamma_{t,s^-} (Z_s+Y_s \beta_s) dW_s+ \Gamma_{t,s^-}  dM_s.
\end{eqnarray*}
From the assumptions made on the coefficients, we obtain that $\phi$ is a martingale and thus
$$\phi_t = Y_t  = \E \left[ \phi_T \bigg| \F_t \right] =\E \left[ Y_T\Gamma_{t,T} + \int_t^T  \Gamma_{t,r} f_r dr \bigg| \F_t \right]  .$$

\end{proof}

The next proposition is a modification of Theorem 4.2 in \cite{quen:sule:13} (see also Theorem 252 in \cite{situ:05}).
\begin{Prop} \label{prop:comp_sol_gene_BSDE}
We consider a generator $f_1$ satisfying ${\bf (H_{ex})}$ and we ask $f_2$ to verify ${\bf (H_{comp})}$. Let $\xi^1$ and $\xi^2$ be two terminal conditions for BSDEs \eqref{eq:gene_BSDE} driven respectively by $f_1$ and $f_2$. Denote by $(Y^1,Z^1,\psi^1,M^1)$ and $(Y^2,Z^2,\psi^2,M^2)$ the respective solutions in some space $\cE^p(0,T)$ with $p>1$. If $\xi^1 \leq \xi^2$ and $f_1(t,Y^1_t,Z^1_t,\psi^1_t) \leq f_2(t,Y^1_t,Z^1_t,\psi^1_t)$, then a.s. for any $t \in [0,T]$, $Y^1_t \leq Y^2_t$.
\end{Prop}
\begin{proof}
As usual we set
$$\what Y = Y^2-Y^1, \quad \what Z = Z^2-Z^1, \quad \what \psi = \psi^2 - \psi^1, \quad \what M = M^2 - M^1.$$
Then $(\what Y,\what Z,\what \psi,\what M)$ satisfies:
\begin{equation*}
\what Y_t = \what \xi + \int_t^T h_s ds - \int_t^T\int_\cU \what \psi_s(u) \tpi(du,ds) -\int_t^T \what Z_s dW_s- \int_t^T d\what M_s,
\end{equation*}
where
$$h_s = f_2(Y^2_s,Z^2_s,\psi^2_s) -f_1(Y^1_s,Z^1_s,\psi^1_s).$$
Now we define
\begin{eqnarray*}
f_s & = & f_2(Y^1_s,Z^1_s,\psi^1_s) -f_1(Y^1_s,Z^1_s,\psi^1_s) \\
\alpha_s & = & \frac{f_2(Y^2_s,Z^1_s,\psi^1_s) - f_2(Y^1_s,Z^1_s,\psi^1_s)}{\what Y_s} \1_{\what Y_s \neq 0} \\
\beta_s & = & \frac{f_2(Y^2_s,Z^2_s,\psi^1_s) - f_2(Y^2_s,Z^1_s,\psi^1_s)}{\what Z_s} \1_{\what Z_s \neq 0}
\end{eqnarray*}
then
\begin{eqnarray*}
h_s & = &f_s + \alpha_s \what Y_s + \beta_s \what Z_s +f_2(Y^2_s,Z^2_s,\psi^2_s) - f_2(Y^2_s,Z^2_s,\psi^1_s)  \\
& \geq& f_s + \alpha_s \what Y_s + \beta_s \what Z_s +\int_\cU \kappa_s^{Y^2_s,Z^2_s,\psi^1_s,\psi^2_s} \what \psi_s(u) \mu(du)
\end{eqnarray*}
since $f_2$ satisfies ${\bf (H_{comp})}$. Moreover since $f_2$ is Lipschitz continuous w.r.t.\ $z$, $|\beta|$ is bounded by $K$, whereas from Assumption (H1), $\alpha$ is bounded from above. Moreover, the process $ \kappa_s^{Y^2_s,Z^2_s,\psi^1_s,\psi^2_s}$ is controlled by $\vartheta \in L^2_\mu$. Therefore the process $\Gamma$ defined by \eqref{eq:doleans_dade_exp} is $q$-integrable for any $q\geq 2$ and
$$\what Y_t \geq \E \left[ \Gamma_{t,T} \what \xi + \int_t^T \Gamma_{t,s} f_s ds \bigg| \F_t \right].$$
To conclude recall that since $-1 \leq \kappa^{y,z,\psi,\phi}_t(u)$, $\Gamma_{t,s} \geq 0$ a.s. and by assumptions, $\what \xi \geq 0$ and $f_s \geq 0$. Therefore $\what Y_t \geq 0$ and the conclusion follows.
\end{proof}

Note that the conditions ${\bf (H_{ex})}$ are just imposed on $f^1$ to ensure existence of a solution $(Y^1,Z^1,\psi^1,M^1)$. This proposition gives again uniqueness of the solution.
\begin{Coro} \label{coro:uniq_sol}
Assume ${\bf (H_{comp})}$ and (H4) (resp. (H5)). Then there exists at most one solution $(Y,Z,\psi,M)$ of BSDE $(\xi,f)$ in $\cE^2(0,T)$ (resp. $\cE^p(0,T)$).
\end{Coro}

\section{Random terminal times}\label{rand_term_time}

We come back to the general multidimensional case but we assume that $\tau$ is a stopping time for the filtration $\mathbb F$, which need not be bounded. Assumptions  ${\bf (H_{ex})}$ still hold with a monotonicity constant $\al$ and a Lipschitz constant $K$. (H2) is replaced by:
\begin{equation}\tag{H2''}
\forall r > 0,\ \forall n \in \N,\quad \sup_{|y|\leq r} (|f(t,y,0,0)-f(t,0,0,0)|) \in L^1(\Omega \times (0,n)).
\end{equation}
We assume that $1< p$ and condition (H4) (or (H4')) is replaced by the following one: for some $\rho \in \R$ such that 
$$\rho > \nu =\al+ \frac{K^2}{(p-1)\wedge 1},$$
we have
\begin{equation}\tag{H5'}
\E \left[ e^{p\rho \tau} |\xi|^p + \int_0^\tau e^{p\rho t} |f(t,0,0,0)|^p dt \right] < +\infty.
\end{equation}
The constant $\al$ appears in (H1) and $K$ in (H3). We will need the following additional assumption
\begin{equation}\tag{H6}
\xi \mbox{ is } \F_\tau-\mbox{measurable and } \E\left[ \int_0^\tau e^{p\rho t}|f(t, \xi_t ,\eta_t, \gamma_t)|^p dt \right] < +\infty,
\end{equation}
where $ \xi_t = \E ( \xi|\F_t) $ and $( \eta, \gamma, N)$ are given by the martingale representation:
$$ \xi = \E ( \xi)  + \int_0^\infty  \eta_s dW_s + \int_0^\infty \int_\cU  \gamma_s(u) \tpi(du,ds) +  N_\tau$$
with
$$\E \left[ \left( \int_0^\infty | \eta_s|^2 ds +  \int_0^\infty \int_\cU | \gamma_s(u)|^2 \mu(du) ds +[N]_\tau \right)^{p/2} \right] < +\infty.$$

\begin{Def}
A process $(Y,Z,\psi,M) = (Y_t,Z_t,\psi_t,M_t)_{t\geq 0}$, such that $Y$ is progressively measurable and \cad and $(Z,\psi,M)\in \cD(0,T)\times \cP \times \cM_{loc}$,  with values in $\R^d \times \R^{d\times k} \times \R^d \times \R^d$ is a solution to the BSDE \eqref{eq:gene_BSDE} with random terminal time $\tau$ with data $(\xi;f)$ if on the set $\{t\geq \tau \}$ $Y_t =\xi$ and $Z_t =\psi_t=M_t=0$, $\P$-a.s., $t\mapsto f(t,Y_t,Z_t,\psi_t) \1_{t\leq T}$ belongs to $L^1_{loc}(0,\infty)$  for any $T\geq 0$, $Z$ belongs to $L^2_{loc}(W)$, $\psi$ belongs to $G_{loc}(\pi)$ and, $\P$-a.s., for all $0\leq t \leq T$,
\begin{eqnarray} \nonumber
Y_{t\wedge \tau} & = & Y_{T\wedge \tau} + \int_{t\wedge \tau}^{T\wedge \tau} f(s,Y_s, Z_s,\psi_s) ds -\int_{t\wedge \tau}^{T\wedge \tau} Z_sdW_s \\ \label{eq:gene_BSDE_rand_time}
& -& \int_{t\wedge \tau}^{T\wedge \tau} \int_\cU \psi_s(u) \tpi(du,ds) - \int_{t\wedge \tau}^{T\wedge \tau} dM_s.
\end{eqnarray}
\end{Def}

\begin{Prop} \label{prop:uniq_random_time}
Under conditions (H1), (H2''), (H3), (H5') and (H6), the BSDE \eqref{eq:gene_BSDE_rand_time} has at most one solution satisfying
\begin{eqnarray} \nonumber
&&\E \left[ e^{p\rho (t\wedge \tau)} |Y_{t\wedge \tau}|^p + \int_{0}^{T\wedge \tau}  e^{p\rho s} |Y_{s}|^{p}  ds +  \int_{0}^{T\wedge \tau}  e^{p\rho s} |Y_{s}|^{p-2} |Z_s|^2 \1_{Y_s\neq 0} ds \right] \\ \nonumber
&& + \E \left[ \int_{0}^{T\wedge \tau} e^{p\rho s}  |Y_s|^{p-2} \1_{Y_s \neq 0} \|\psi_s\|^2_{L^2_\mu} ds + \int_{0}^{T\wedge \tau} e^{p\rho s}  |Y_{s}|^{p-2}  \1_{Y_s\neq 0} d[ M]^c_s \right] \\  \nonumber
&& + \E\left[  \sum_{0 < s \leq T\wedge  \tau}e^{p\rho s} |\Delta M_s|^2  \left( |Y_{s^-}|^2 \vee  |Y_{s^-} + \Delta M_s|^2 \right)^{p/2-1} \1_{|Y_{s^-}| \vee  |Y_{s^-} + \Delta M_s| \neq 0} \right] \\ \label{eq:rnd_term_time_apriori_estim}
&& \quad  < +\infty. 
\end{eqnarray}
\end{Prop}
\begin{proof}
Assume that there exist two solutions $(Y,Z,\psi,M)$ and $(Y',Z',\psi',M')$ satisfying \eqref{eq:rnd_term_time_apriori_estim} and let
$$\what Y_t = Y_t - Y'_t, \quad \what Z_t = Z_t - Z'_t, \quad \what \psi_t = \psi_t - \psi'_t, \quad \what M_t = M_t - M'_t.$$
Let us denote $c(p)=\frac{p}{2}((p-1)\wedge 1)$. From Corollary \ref{coro:ito_form_xp}, Lemma \ref{lem:control_jumps_part} and Remark \ref{rem:p_geq_2} we have for $0 \leq t \leq T$
\begin{eqnarray} \nonumber
&&e^{p\rho (t\wedge \tau)} |\what Y_{t\wedge \tau}|^p + c(p) \int_{t\wedge \tau}^{T\wedge \tau}  e^{p\rho s} |\what Y_{s}|^{p-2} |\what Z_s|^2 \1_{\what Y_s\neq 0} ds +c(p) \int_{t\wedge \tau}^{T\wedge \tau} e^{p\rho s}  |\what Y_{s}|^{p-2}  \1_{\what Y_s\neq 0} d[ \what M ]^c_s\\ \nonumber
&& + c(p)  \sum_{t\wedge \tau < s \leq T\wedge \tau}e^{p\rho s}  \left( |\what Y_{s^-}|^2 \vee  |\what Y_{s^-} + \Delta \what M_s|^2 \right)^{p/2-1} \1_{|\what Y_{s^-}| \vee  |\what Y_{s^-} + \Delta \what M_s| \neq 0} |\Delta \what M_s|^2  \\ \nonumber
&& + c(p)  \int_{t\wedge \tau}^{T\wedge \tau}\int_\cU e^{p\rho s}  \left( |\what Y_{s^-}|^2 \vee  |\what Y_{s^-} + \what \psi_s(u)|^2 \right)^{p/2-1} \1_{|\what Y_{s^-}| \vee  |\what Y_{s^-} + \what \psi_s(u)| \neq 0}|\what \psi_s(u)|^2\pi(du,ds) \\ \nonumber
&& \leq e^{p\rho (T\wedge \tau)}|\what Y_{T\wedge \tau}|^p \\ \nonumber
&& \quad + p \int_{t\wedge \tau}^{T\wedge \tau} e^{p\rho s}\left( |\what Y_s|^{p-1} \check{\what Y}_s (f(s,Y_s,Z_s,\psi_s)-f(s,Y'_s,Z'_s,\psi'_s)) -\rho |\what Y_s|^p\right) ds  \\ \nonumber
&&\quad -  p \int_{t\wedge \tau}^{T\wedge \tau} e^{p\rho s} |\what Y_s|^{p-1} \check{\what Y}_s \what Z_s dW_s  -  p  \int_{t\wedge \tau}^{T\wedge \tau} e^{p\rho s}|\what Y_{s^-}|^{p-1} \check{\what Y}_{s^-} d\what M_s \\ \label{eq:Ito_p_tau_estimate}
& &\quad  -  p \int_{t\wedge \tau}^{T\wedge \tau}e^{p\rho s} |\what Y_s|^{p-1} \check{\what Y}_s  \int_\cU \what \psi_s(u) \tpi(du,ds) .
\end{eqnarray}
From the assumption on $f$ and Young's inequality we deduce that
\begin{eqnarray*}
&& |\what y|^{p-1} \check{\what y} (f(s,y,z,\psi)-f(s,y',z',\psi') -\rho|\what y|^p \leq  \left( \al  +\frac{K^2}{(p-1)\wedge 1} -\rho \right) |\what y|^p \\
&& \qquad + \frac{(p-1)\wedge 1}{2}  |\what y|^{p-2} \1_{\what y\neq 0}  |\what z|^2   +\frac{(p-1)\wedge 1}{2}  |\what y|^{p-2} \1_{\what y\neq 0} \|\what \psi\|_{L^2} \\
&& \quad \leq \frac{(p-1)\wedge 1}{2}  |\what y|^{p-2} \1_{\what y\neq 0}  |\what z|^2   +\frac{(p-1)\wedge 1}{2}  |\what y|^{p-2} \1_{\what y\neq 0} \|\what \psi\|_{L^2}.
\end{eqnarray*}
Note that from the integrability conditions on the solution every local martingale involved in \eqref{eq:Ito_p_tau_estimate} is a uniformly integrable martingale. Moreover using a localization argument (see Equation \eqref{eq:local_mart_jump}), the two terms:
$$ \int_{t\wedge \tau}^{T\wedge \tau}\int_\cU e^{p\rho s}  \left( |\what Y_{s^-}|^2 \vee  |\what Y_{s^-} + \what \psi_s(u)|^2 \right)^{p/2-1} \1_{|\what Y_{s^-}| \vee  |\what Y_{s^-} + \what \psi_s(u)| \neq 0}|\what \psi_s(u)|^2\pi(du,ds)$$
and 
$$ \int_{t\wedge \tau}^{T\wedge \tau}\int_\cU e^{p\rho s} |\what Y_s|^{p-2} \1_{\what Y_s \neq 0} \|\what \psi_s(u)\|_{L^2}\mu(du) ds$$
have the same expectation. Hence taking the expectation in \eqref{eq:Ito_p_tau_estimate} we obtain:
$$\E e^{p\rho (t\wedge \tau)} |\what Y_{t\wedge \tau}|^p \leq \E e^{p\rho (T\wedge \tau)}|\what Y_{T\wedge \tau}|^p.$$
If we replace $\rho$ by $\rho'$ with $\al  +\frac{K^2}{(p-1)\wedge 1} < \rho'<\rho$ we obtain the same result, and thus we get for any $0\leq t \leq T $
$$\E e^{p\rho' (t\wedge \tau)} |\what Y_{t\wedge \tau}|^p \leq e^{p(\rho'-\rho)T} \E e^{p\rho (T\wedge \tau)}|\what Y_{T\wedge \tau}|^p .$$
We let $T$ go to infinity to obtain $\what Y_t = 0$. 

Therefore $(Y,Z,\psi,M)$ and $(Y',Z',\psi',M')$ satisty BSDE \eqref{eq:gene_BSDE_rand_time} and $Y=Y'$. Thus we have the same martingale parts and by orthogonality, $\what Z=\what \psi =\what M=0$. Uniqueness of the solution is proved. 
\end{proof}

\begin{Prop} \label{prop:exis_Lp_sol_rnd_time}
Under conditions (H1), (H2''), (H3), (H5') and (H6), the BSDE \eqref{eq:gene_BSDE_rand_time} has a solution satisfying 
\begin{eqnarray} \nonumber
&&\E \left[ e^{p\rho (t\wedge \tau)} |Y_{t\wedge \tau}|^p + \int_{0}^{T\wedge \tau}  e^{p\rho s} |Y_{s}|^{p}  ds +  \int_{0}^{T\wedge \tau}  e^{p\rho s} |Y_{s}|^{p-2} |Z_s|^2 \1_{Y_s\neq 0} ds \right] \\ \nonumber
&& + \E \left[ \int_{0}^{T\wedge \tau} e^{p\rho s}  |Y_s|^{p-2} \1_{Y_s \neq 0} \|\psi_s\|^2_{L^2_\mu} ds + \int_{0}^{T\wedge \tau} e^{p\rho s}  |Y_{s}|^{p-2}  \1_{Y_s\neq 0} d[ M]^c_s \right] \\  \nonumber
&& + \E\left[  \sum_{0 < s \leq T\wedge  \tau}e^{p\rho s} |\Delta M_s|^2  \left( |Y_{s^-}|^2 \vee  |Y_{s^-} + \Delta M_s|^2 \right)^{p/2-1} \1_{|Y_{s^-}| \vee  |Y_{s^-} + \Delta M_s| \neq 0} \right] \\ \label{eq:rnd_term_time_apriori_estim_2}
&& \leq C \E\left[  e^{p\rho \tau} |\xi|^p  +  \int_{0}^{\tau} e^{p\rho s} |f(s,0,0,0)|^p ds \right]. 
\end{eqnarray}
Moreover 
\begin{eqnarray} \nonumber
&& \E \left( \int_{0}^{\tau} e^{2\rho s} |Z_s|^2 ds \right)^{p/2}+ \E\left(  \int_{0}^{ \tau} e^{2\rho s}\int_\cU |\psi_s(u)|^2 \mu(du) ds  \right)^{p/2}+ \E \left( \int_{0}^{ \tau} e^{2\rho s} d[M]_s\right)^{p/2}\\ \label{eq:rnd_term_time_apriori_estim_3}
&& \qquad  \leq C \E\left[  e^{p\rho \tau} |\xi|^p  +  \int_{0}^{\tau} e^{p\rho s} |f(s,0,0,0)|^p ds \right]. 
\end{eqnarray}
The constant $C$ depends only on $p$, $K$ and $\alpha$. 
\end{Prop}
\begin{proof}
We follow the proof of Theorem 4.1 in \cite{pard:99}. For each $n\in \N$ we construct a solution $\{(Y^n,Z^n,\psi^n,M^n), \ t\geq 0\}$ as follows. By Theorem \ref{thm:exist_uniq_sol_p_leq_2}, on the interval $[0,n]$:
\begin{eqnarray} \nonumber
Y^n_{t} & = & \E (\xi | \F_n)  + \int_{t}^{n} \1_{[0,\tau]}(s) f(s,Y^n_s, Z^n_s,\psi^n_s) ds -\int_{t}^{n} Z^n_sdW_s \\ \nonumber
& -& \int_{t}^{n} \int_\cU \psi^n_s(u) \tpi(du,ds) - \int_{t}^{n} dM^n_s.
\end{eqnarray}
And for $t \geq n$ (Assumption (H6)):
$$Y^n_t = \xi_t ,\quad Z^n_t = \eta_t, \quad \psi^n_t(u) = \gamma_t(u), \quad M^n_t = N_t.$$

\begin{itemize}
\item \noindent \textbf{Step 1:} \textit{the sequence $\{(Y^n,Z^n,\psi^n,M^n), \ t\geq 0\}$ satisfies Inequality \eqref{eq:rnd_term_time_apriori_estim_2}.}
\end{itemize}
Using Corollary \ref{coro:ito_form_xp}, Lemma \ref{lem:control_jumps_part} and Remark \ref{rem:p_geq_2} we have for $0 \leq t \leq T \leq n$
\begin{eqnarray*}
&&e^{p\rho (t\wedge \tau)} |Y^n_{t\wedge \tau}|^p + c(p) \int_{t\wedge \tau}^{T\wedge \tau}  e^{p\rho s} |Y^n_{s}|^{p-2} |Z^n_s|^2 \1_{Y^n_s\neq 0} ds +c(p) \int_{t\wedge \tau}^{T\wedge \tau} e^{p\rho s}  |Y^n_{s}|^{p-2}  \1_{Y^n_s\neq 0} d[ M^n ]^c_s\\
&& + c(p)  \sum_{t\wedge \tau < s \leq T\wedge \tau}e^{p\rho s} |\Delta M^n_s|^2  \left( |Y^n_{s^-}|^2 \vee  |Y^n_{s^-} + \Delta M^n_s|^2 \right)^{p/2-1} \1_{|Y^n_{s^-}| \vee  |Y^n_{s^-} + \Delta M^n_s| \neq 0} \\
&& + c(p)  \int_{t\wedge \tau}^{T\wedge \tau}\int_\cU e^{p\rho s}|\psi^n_s(u)|^2  \left( |Y^n_{s^-}|^2 \vee  |Y^n_{s^-} + \psi^n_s(u)|^2 \right)^{p/2-1} \1_{|Y^n_{s^-}| \vee  |Y^n_{s^-} + \psi^n_s(u)|\neq 0}\pi(du,ds) \\
&& \leq e^{p\rho (T\wedge \tau)}|Y^n_{T\wedge \tau}|^p + p \int_{t\wedge \tau}^{T\wedge \tau} e^{p\rho s}\left( |Y^n_s|^{p-1} \check{Y}^n_s f(s,Y^n_s,Z^n_s,\psi^n_s) -\rho |Y^n_s|^p\right) ds  \\
&&\quad -  p \int_{t\wedge \tau}^{T\wedge \tau} e^{p\rho s} |Y^n_s|^{p-1} \check{Y}^n_s Z^n_s dW_s  -  p  \int_{t\wedge \tau}^{T\wedge \tau} e^{p\rho s}|Y^n_{s^-}|^{p-1} \check{Y}^n_{s^-} dM^n_s \\
& &\quad  -  p \int_{t\wedge \tau}^{T\wedge \tau}e^{p\rho s} |Y^n_s|^{p-1} \check{Y}^n_s  \int_\cU \psi^n_s(u) \tpi(du,ds) .
\end{eqnarray*}
Now with Young's inequality and for some $\delta > 0$ sufficiently small
\begin{eqnarray} \nonumber
&& |y|^{p-1}\check{y} f(t,y,z,\psi) \leq \left( \al +\delta +\frac{K^2}{((p-1)\wedge 1-2\delta)} \right) |y|^p \\ \nonumber
&& \qquad + \left( \frac{(p-1)\wedge 1}{2} -\delta \right) |y|^{p-2} \1_{y\neq 0}  |z|^2  + \frac{1}{p} |f(t,0,0,0)|^p \left( \frac{p\delta }{(p-1)\wedge 1}\right)^{1-p} \\ \label{eq:estim_12_10_2014_2}
&& \qquad + \left( \frac{(p-1)\wedge 1}{2} -\delta \right) |y|^{p-2} \1_{y\neq 0} \|\psi\|_{L^2}  .
\end{eqnarray}
We choose $\delta > 0$ such that $ \al +2\delta +\frac{K^2}{((p-1)\wedge 1-2\delta)} \leq \rho$ and we obtain:
\begin{eqnarray*} \nonumber
&&e^{p\rho (t\wedge \tau)} |Y^n_{t\wedge \tau}|^p +p\delta   \int_{t\wedge \tau}^{T\wedge \tau}  e^{p\rho s} |Y^n_{s}|^{p}  ds + p\delta \int_{t\wedge \tau}^{T\wedge \tau}  e^{p\rho s} |Y^n_{s}|^{p-2} |Z^n_s|^2 \1_{Y^n_s\neq 0} ds \\ \nonumber
&& +c(p) \int_{t\wedge \tau}^{T\wedge \tau} e^{p\rho s}  |Y^n_{s}|^{p-2}  \1_{Y^n_s\neq 0} d[ M^n ]^c_s\\ \nonumber
&& +c(p) \sum_{t\wedge \tau < s \leq T\wedge \tau}e^{p\rho s} |\Delta M^n_s|^2  \left( |Y^n_{s^-}|^2 \vee  |Y^n_{s^-} + \Delta M^n_s|^2 \right)^{p/2-1} \1_{|Y^n_{s^-}| \vee  |Y^n_{s^-} + \Delta M^n_s|\neq 0} \\ \nonumber
&& +c(p) \int_{t\wedge \tau}^{T\wedge \tau}\int_\cU e^{p\rho s}|\psi^n_s(u)|^2  \left( |Y^n_{s^-}|^2 \vee  |Y^n_{s^-} + \psi^n_s(u)|^2 \right)^{p/2-1} \1_{|Y^n_{s^-}| \vee  |Y^n_{s^-} + \psi^n_s(u)| \neq 0}\pi(du,ds) \\ \nonumber
&& - p\left( \frac{(p-1)\wedge 1}{2} -\delta \right) \int_{t\wedge \tau}^{T\wedge \tau} e^{p\rho s}  |Y^n_s|^{p-2} \1_{Y^n_s \neq 0} \|\psi^n_s\|_{L^2} ds \\ \nonumber
&& \leq e^{p\rho (T\wedge \tau)} |Y_{T\wedge \tau}|^p  + \int_{t\wedge \tau}^{T\wedge \tau} e^{p\rho s} |f(s,0,0,0)|^p \left( \frac{p\delta }{(p-1)\wedge 1}\right)^{1-p}  ds \\ \nonumber
&& \quad -  p \int_{t\wedge \tau}^{T\wedge \tau} e^{p\rho s}|Y^n_s|^{p-1} \check{Y}^n_s Z^n_s dW_s  \\ \label{eq:estim_12_10_2014}
& &\quad  -  p  \int_{t\wedge \tau}^{T\wedge \tau} e^{p\rho s}|Y^n_{s^-}|^{p-1} \check{Y}^n_{s^-} dM^n_s -  p \int_{t\wedge \tau}^{T\wedge \tau}e^{p\rho s} |Y^n_s|^{p-1} \check{Y}^n_s  \int_\cU \psi^n_s(u) \tpi(du,ds) .
\end{eqnarray*}
Taking the expectation we get
\begin{eqnarray} \nonumber
&&\E \left[ e^{p\rho (t\wedge \tau)} |Y^n_{t\wedge \tau}|^p +p\delta   \int_{0}^{T\wedge \tau}  e^{p\rho s} |Y^n_{s}|^{p}  ds\right] \\ \nonumber
&& +p\delta \E \left[ \int_{0}^{T\wedge \tau} e^{p\rho s}  |Y^n_s|^{p-2} \1_{Y^n_s \neq 0} \|\psi^n_s\|_{L^2} ds + \int_{0}^{T\wedge \tau}  e^{p\rho s} |Y^n_{s}|^{p-2} |Z^n_s|^2 \1_{Y^n_s\neq 0} ds \right] \\ \nonumber
&& +c(p)\E  \int_{0}^{T\wedge \tau} e^{p\rho s}  |Y^n_{s}|^{p-2}  \1_{Y^n_s\neq 0} d[ M^n ]^c_s\\ \nonumber
&& + c(p) \E\left[  \sum_{0 < s \leq T\wedge  \tau}e^{p\rho s} |\Delta M^n_s|^2  \left( |Y^n_{s^-}|^2 \vee  |Y^n_{s^-} + \Delta M^n_s|^2 \right)^{p/2-1} \1_{|Y^n_{s^-}| \vee  |Y^n_{s^-} + \Delta M^n_s| \neq 0} \right] \\ \nonumber
&& + c(p) \E \int_{t\wedge \tau}^{T\wedge \tau}e^{p\rho s} \int_\cU |\psi^n_s(u)|^2  \\ \nonumber
&& \qquad \qquad \left[ \left( |Y^n_{s^-}|^2 \vee  |Y^n_{s^-} + \psi^n_s(u)|^2 \right)^{p/2-1} \1_{|Y^n_{s^-}| \vee  |Y^n_{s^-} + \psi^n_s(u)| \neq 0}\pi(du,ds) - \frac{1}{2}|Y^n_s|^{p-2} \1_{Y^n_s \neq 0} \mu(du)ds\right]  \\
\label{eq:estim_12_10_2014_3}
&& \leq \E\left[  e^{p\rho(T\wedge \tau)} |Y^n_{T\wedge \tau}|^p  + \left( \frac{p\delta }{(p-1)\wedge 1}\right)^{1-p}  \int_{0}^{T\wedge \tau} e^{p\rho s} |f(s,0,0,0)|^p ds \right].
\end{eqnarray}
Using an argument based on Burkholder-Davis-Gundy inequality (see Step 1 in the proof of Proposition \ref{prop:Lp_estimates} ($p\geq 2$) or Proposition \ref{prop:apriori_lp} ($p < 2$)) we can moreover include a $\sup_{t\in [0,n]}$ inside the expectation on the left-hand side.

\begin{itemize}
\item \noindent \textbf{Step 2:} \textit{the sequence $(Y^n)$ converges. }
\end{itemize}
Take $m > n$ and define
$$\what Y_t = Y^m_t - Y^n_t, \quad \what Z_t = Z^m_t - Z^n_t, \quad \what \psi_t = \psi^m_t - \psi^n_t, \quad \what M_t = M^m_t - M^n_t.$$
For $n \leq t \leq m$,
\begin{eqnarray*}
\what Y_t & = & \int_{t\wedge \tau}^{m\wedge \tau} f(s,Y^m_s,Z^m_s,\psi^m_s) ds - \int_{t\wedge \tau}^{m\wedge \tau} \what Z_s dW_s -  \int_{t\wedge \tau}^{m\wedge \tau} \int_\cU \what \psi_s(u) \tpi(du,ds)  \\
&& \quad - \what M_{m\wedge \tau} + \what M_{t\wedge \tau}.
\end{eqnarray*}
Thus for $n \leq t \leq m$,
\begin{eqnarray*}
&&e^{p\rho (t\wedge \tau)} |\what Y_{t\wedge \tau}|^p + c(p) \int_{t\wedge \tau}^{m\wedge \tau}  e^{p\rho s} |\what Y_{s}|^{p-2} |\what Z_s|^2 \1_{\what Y_s\neq 0} ds +c(p) \int_{t\wedge \tau}^{m\wedge \tau} e^{p\rho s}  |\what Y_{s}|^{p-2}  \1_{\what Y_s\neq 0} d[ \what M ]^c_s\\
&& + c(p) \sum_{t\wedge \tau < s \leq m\wedge \tau}e^{p\rho s} |\Delta \what M_s|^2  \left( |\what Y_{s^-}|^2 \vee  |\what Y_{s^-} + \Delta \what M_s|^2 \right)^{p/2-1} \1_{ |\what Y_{s^-}| \vee  |\what Y_{s^-} + \Delta \what M_s| \neq 0} \\
&& + c(p) \int_{t\wedge \tau}^{ m \wedge \tau}\int_\cU e^{p\rho s}|\what \psi_s(u)|^2  \left( |\what Y_{s^-}|^2 \vee  |\what Y_{s^-} + \what \psi_s(u)|^2 \right)^{p/2-1} \1_{ |\what Y_{s^-}| \vee  |\what Y_{s^-} + \what \psi_s(u)| \neq 0}\pi(du,ds) \\
&& \leq  p \int_{t\wedge \tau}^{m\wedge \tau} e^{p\rho s}\left( |\what Y_s|^{p-1} \check{\what Y}_s f(s,Y^m_s,Z^m_s,\psi^m_s) -\rho |\what Y_s|^p\right) ds  \\
&&\quad -  p \int_{t\wedge \tau}^{m\wedge \tau} e^{p\rho s} |\what Y_s|^{p-1} \check{\what Y}_s \what Z_s dW_s  -  p  \int_{t\wedge \tau}^{m\wedge \tau} e^{p\rho s}|\what Y_{s^-}|^{p-1} \check{\what Y}_{s^-} d\what M_s \\
& &\quad  -  p \int_{t\wedge \tau}^{m\wedge \tau}e^{p\rho s} |\what Y_s|^{p-1} \check{\what Y}_s  \int_\cU \what \psi_s(u) \tpi(du,ds) \\
&& \leq  p \int_{t\wedge \tau}^{m\wedge \tau} e^{p\rho s}\left( \al |\what Y_s|^{p} + K |\what Y_s|^{p-1} |\what Z_s| + K |\what Y_s|^{p-1} \|\what \psi_s\|_{L^2}  -\rho |\what Y_s|^p\right) ds  \\
&& \quad + p \int_{t\wedge \tau}^{m\wedge \tau} e^{p\rho s} |\what Y_s|^{p-1} \check{\what Y}_s f(s,\xi_s,\eta_s,\gamma_s) ds \\
&&\quad -  p \int_{t\wedge \tau}^{m\wedge \tau} e^{p\rho s} |\what Y_s|^{p-1} \check{\what Y}_s \what Z_s dW_s  -  p  \int_{t\wedge \tau}^{m\wedge \tau} e^{p\rho s}|\what Y_{s^-}|^{p-1} \check{\what Y}_{s^-} d\what M_s \\
& &\quad  -  p \int_{t\wedge \tau}^{m\wedge \tau}e^{p\rho s} |\what Y_s|^{p-1} \check{\what Y}_s  \int_\cU \what \psi_s(u) \tpi(du,ds) .
\end{eqnarray*}
By an argument already used to control the generator (see \eqref{eq:estim_12_10_2014_2}) and to obtain Inequality \eqref{eq:estim_12_10_2014}, we deduce that
\begin{eqnarray}  \nonumber
&&\E \left[ \sup_{t\in [n,m]} e^{p\rho (t\wedge \tau)} |\what Y_{t\wedge \tau}|^p + \int_{n\wedge \tau}^{m\wedge \tau} e^{p\rho s} |\what Y_s|^p ds +  \int_{n\wedge \tau}^{m\wedge \tau} e^{p\rho s}  |\what Y_{s}|^{p-2}  \1_{\what Y_s\neq 0} d[ \what M ]^c_s\right] \\  \nonumber
&& + \E \left[  \int_{n\wedge \tau}^{m\wedge \tau}  e^{p\rho s} |\what Y_{s}|^{p-2} |\what Z_s|^2 \1_{\what Y_s\neq 0} ds +  \int_{n\wedge \tau}^{m\wedge \tau} \int_\cU e^{p\rho s} |\what Y_{s}|^{p-2} |\what \psi_s(u)|^2 \1_{\what Y_s\neq 0}\mu(du) ds \right] \\  \nonumber
&& + \E \left[  \sum_{n\wedge \tau < s \leq m\wedge \tau}e^{p\rho s} |\Delta \what M_s|^2  \left( |\what Y_{s^-}|^2 \vee  |\what Y_{s^-} + \Delta \what M_s|^2 \right)^{p/2-1} \1_{\what Y_{s^-} \neq 0} \right] \\ \label{eq:estim_tau_convergence}
&&\qquad  \leq  C\E  \int_{n\wedge \tau}^{\tau} e^{p\rho s}|f(s,\xi_s,\eta_s,\gamma_s)|^p ds .
\end{eqnarray}
By assumption the last term goes to zero as $n$ goes to infinity. Next for $t \leq n$
\begin{eqnarray*}
\what Y_t & = & \what Y_n + \int_{n\wedge \tau}^{m\wedge \tau} (f(s,Y^m_s,Z^m_s,\psi^m_s)-f(s,Y^n_s,Z^n_s,\psi^n_s)) ds - \int_{n\wedge \tau}^{m\wedge \tau} \what Z_s dW_s \\
&& \quad -  \int_{n\wedge \tau}^{m\wedge \tau} \int_\cU \what \psi_s(u) \tpi(du,ds) - \what M_{m\wedge \tau} + \what M_{n\wedge \tau}.
\end{eqnarray*}
It follows from the same argument as in the proof of Proposition \ref{prop:uniq_random_time} that
\begin{eqnarray*}
\E e^{p\rho (t\wedge \tau)} |\what Y_{t\wedge \tau}|^p +\E  \int_0^{\tau} e^{p\rho s} |\what Y_s|^p  ds &  \leq & \E e^{p\rho (n\wedge \tau)}|\what Y_{n}|^p \\
& \leq &  C\E  \int_{n\wedge \tau}^{\tau} e^{p\rho s}|f(s,\xi_s,\eta_s,\gamma_s)|^p ds 
\end{eqnarray*}
and the convergence of the sequence $Y^n$. Moreover from the first step the limit satisfies the a priori estimate \eqref{eq:rnd_term_time_apriori_estim_2}.

\begin{itemize}
\item \noindent \textbf{Step 3:} \textit{convergence of the martingale part $(Z^n ,\psi^n,M^n)$.} 
\end{itemize}
The proof is rather different for $p\geq 2$ and $p< 2$. In the first case, we follow the proof of \cite{pard:99}, Theorem 4.1 and Proposition \ref{prop:Lp_estimates}. We apply It\^o's formula to $e^{2\rho s} |\what Y_s|^2$ for $n \leq t \leq m$:
\begin{eqnarray*} 
&&  \int_{t\wedge \tau}^{m\wedge \tau} e^{2\rho s} |\what Z_s|^2 ds +  \int_{t\wedge \tau}^{m\wedge \tau} \int_\cU e^{2\rho s} |\what \psi_s(u)|^2 \pi(du,ds) + \int_{t\wedge \tau}^{m\wedge \tau}e^{2\rho s} d\what  M_{s} \\ 
&& \quad =  2 \int_{t\wedge \tau}^{m\wedge \tau}e^{2\rho s} \what Y_s \left( f(s,Y^m_s,Z^m_s,\psi^m_s) -f(s,\xi_s,\eta_s,\gamma_s) -  \rho |\what Y_s|^2 \right) ds \\
&& \qquad + 2 \int_{t\wedge \tau}^{m\wedge \tau}e^{2\rho s} \what Y_s f(s,\xi_s,\eta_s,\gamma_s)ds\\
&& \qquad -2 \int_{t\wedge \tau}^{m\wedge \tau} e^{2\rho s}\what Y_{s} \what Z_s dW_s  -2\int_{t\wedge \tau}^{m\wedge \tau}e^{2\rho s} \what Y_{s^-} d\what M_s  - 2 \int_{t\wedge \tau}^{m\wedge \tau}e^{2\rho s} \int_\cU \what Y_{s^-} \what \psi_s(u)  \tpi(du,ds) .
\end{eqnarray*}
With the same arguments used to obtain \eqref{eq:Lp_estim_mart_part_p_geq _2}, the assumptions on $f$, Burkholder-Davis-Gundy ($p/2 \geq 1$) and Young's inequality lead to:
\begin{eqnarray*}  \nonumber
&& \E \left( \int_{n\wedge \tau}^{m\wedge \tau} e^{2\rho s} |\what Z_s|^2 ds \right)^{p/2}+ \E\left(  \int_{n\wedge \tau}^{m\wedge \tau} e^{2\rho s}\int_\cU |\what \psi_s(u)|^2 \mu(du) ds  \right)^{p/2}+ \E \left( \int_{n\wedge \tau}^{m\wedge \tau} e^{2\rho s} d\what  M_s\right)^{p/2}\\  
&& \quad \leq \widetilde C_{p,K,T,\eps} \ \E \left( \sup_{n\leq t} e^{p\rho s}  |\what Y_s|^p\right)  + \widetilde C_p \E \left[  \int_{n\wedge \tau}^{\tau} e^{p\rho s} | f(s,\xi_s,\eta_s,\gamma_s)|^p ds \right]   .
\end{eqnarray*}
Then following the same scheme but with $t \leq n$, we obtain:
\begin{eqnarray*}  \nonumber
&& \E \left( \int_{0}^{n\wedge \tau} e^{2\rho s} |\what Z_s|^2 ds \right)^{p/2}+ \E\left(  \int_{0}^{n\wedge \tau} e^{2\rho s}\int_\cU |\what \psi_s(u)|^2 \mu(du) ds  \right)^{p/2}+ \E \left( \int_{0}^{n\wedge \tau} e^{2\rho s} d\what  M_s\right)^{p/2}\\  
&& \quad \leq \widetilde C_{p,K,T,\eps} \ \E \left( e^{p\rho (n\wedge \tau)}  |\what Y_{n\wedge \tau}|^p\right) .
\end{eqnarray*}

Now assume that $p\in (1,2)$. From Inequality \eqref{eq:estim_tau_convergence} and by the proof of uniqueness we deduce that 
\begin{eqnarray*}  \nonumber
&&\E \left[  \int_{0}^{m\wedge \tau} e^{p\rho s}  |\what Y_{s}|^{p-2}  \1_{\what Y_s\neq 0} d[ \what M ]^c_s+  \int_{0}^{m\wedge \tau}  e^{p\rho s} |\what Y_{s}|^{p-2} |\what Z_s|^2 \1_{\what Y_s\neq 0} ds\right] \\  \nonumber
&& + \E \left[   \int_{0}^{m\wedge \tau} \int_\cU e^{p\rho s} |\what Y_{s}|^{p-2} |\what \psi_s(u)|^2 \1_{\what Y_s\neq 0}\mu(du) ds \right] \\  \nonumber
&& + \E \left[ \int_{0}^{ m \wedge \tau}\int_\cU e^{p\rho s}|\what \psi_s(u)|^2  \left( |\what Y_{s^-}|^2 \vee  |\what Y_{s^-} + \what \psi_s(u)|^2 \right)^{p/2-1} \1_{ |\what Y_{s^-}| \vee  |\what Y_{s^-} + \what \psi_s(u)| \neq 0}\pi(du,ds) \right] \\ 
&& + \E \left[  \sum_{0 < s \leq m\wedge \tau}e^{p\rho s} |\Delta \what M_s|^2  \left( |\what Y_{s^-}|^2 \vee  |\what Y_{s^-} + \Delta \what M_s|^2 \right)^{p/2-1} \1_{|\what Y_{s^-}| \vee  |\what Y_{s^-} + \Delta \what M_s| \neq 0} \right] \\ 
&&\qquad  \leq  C\E  \int_{n\wedge \tau}^{\tau} e^{p\rho s}|f(s,\xi_s,\eta_s,\gamma_s)|^p ds .
\end{eqnarray*}
Then we can use again the argument \eqref{eq:trick_control_mart_part} in order to have:
\begin{eqnarray*}
&& \E \left( \int_0^{m\wedge \tau} e^{2 \rho s} |\what Z_s|^2 ds \right)^{p/2} = \E \left( \int_0^{m\wedge \tau} e^{2 \rho s} \1_{\what Y_s \neq 0} |\what Z_s|^2 ds \right)^{p/2} \\
&& \quad \leq \frac{2-p}{2} \E \left[ \sup_{n\leq t}\left(  e^{\rho ps}|\what Y_{s}|^{p} \right)\right] + \frac{p}{2}\E \int_0^T e^{\rho ps} |\what Y_{s}|^{p-2} \1_{\what Y_s\neq 0} |\what Z_s|^2 ds\\
&& \quad \leq \frac{2-p}{2} \E \left[ \sup_{n\leq t} \left( e^{\rho ps}|\what Y_{s}|^{p} \right) \right]+ \frac{p}{2} C\E  \int_{n\wedge \tau}^{\tau} e^{p\rho s}|f(s,\xi_s,\eta_s,\gamma_s)|^p ds.
\end{eqnarray*}
We can repeat this for $\what \psi$ and $\what M$.

Therefore in both cases we proved that the sequence $(Z^n ,\psi^n,M^n)$ is a Cauchy sequence for the norm:
$$\E \left( \int_{0}^{\tau} e^{2\rho s} |\what Z_s|^2 ds \right)^{p/2}+ \E\left(  \int_{0}^{ \tau} e^{2\rho s}\int_\cU |\what \psi_s(u)|^2 \mu(du) ds  \right)^{p/2}+ \E \left( \int_{0}^{ \tau} e^{2\rho s} d[\what  M]_s\right)^{p/2}$$
Hence it converges to $(Z,\psi,M)$ and from the two previous steps the limit $(Y,Z,\psi,M)$ is a solution of the BSDE \eqref{eq:gene_BSDE_rand_time} which satisfies \eqref{eq:rnd_term_time_apriori_estim_2} and \eqref{eq:rnd_term_time_apriori_estim_3}.
\end{proof}

From the two previous propositions we deduce:
\begin{Thm} \label{thm:exis_Lp_sol_rnd_time}
Under conditions (H1), (H2''), (H3), (H5') and (H6), the BSDE \eqref{eq:gene_BSDE_rand_time} has a unique solution satisfying \eqref{eq:rnd_term_time_apriori_estim_2} and \eqref{eq:rnd_term_time_apriori_estim_3}.
\end{Thm}

\begin{Remark}
As in Pardoux \cite{pard:99} (Exercise 4.2), one can replace the condition $\rho > \nu =\al+ \frac{K^2}{(p-1)}$ by the condition $\rho > \al$ if there exists a progressively measurable process $g$ such that for any $z$ and $\psi$
$$|f(t,0,z,\psi)| \leq g_t,$$
and
$$ E\int_0^\tau e^{p\rho t} |g_t|^p dt<\infty.$$
In this case the conclusion of Theorem \ref{thm:exis_Lp_sol_rnd_time} also holds.
\end{Remark}
Indeed for $p\geq 2$ as in the proof of Proposition \ref{prop:Lp_estimates} we can obtain for every $0\leq t \leq T$ and every $\rho > \al$
\begin{eqnarray*}
&& e^{p\rho (t\wedge \tau)} |Y_{t\wedge \tau}|^p +  \kappa_p \int_{t\wedge \tau}^{T \wedge \tau} e^{p\rho s}| Y_{s}|^{p-2} |  Z_s|^2 ds + \kappa_p \int_{t\wedge \tau}^{T \wedge \tau} e^{p\rho s}| Y_{s^-}|^{p-2} d[ M ]_s \\
&& \qquad + \kappa_p\int_{t\wedge \tau}^{T \wedge \tau} e^{p\rho s} |Y_{s^-}  |^{p-2}  \| \psi_s\|_{L^2_\mu}^2 ds \\
&&\quad \leq e^{p\rho (T\wedge \tau)} |Y_{T\wedge \tau}|^p  + \int_{t\wedge \tau}^{T \wedge \tau} p e^{p\rho s} \left( Y_s | Y_s|^{p-2} f(s,Y_s,Z_s,\psi_s) - \rho |Y_s|^p \right) ds\\
&& \quad  -  p  \int_{t\wedge \tau}^{T \wedge \tau}e^{p\rho s}Y_{s^-} |Y_{s^-}|^{p-2} d M_s  -  p  \int_{t\wedge \tau}^{T \wedge \tau}  e^{p\rho s}Y_{s^-} |Y_{s^-}|^{p-2} Z_s dW_s   \\
& &\quad   - \int_{t\wedge \tau}^{T \wedge \tau}e^{p\rho s} \int_\cU \left( |Y_{s^-} + \psi_s(u) |^p - |Y_{s^-}|^p \right) \tpi(du,ds)
\end{eqnarray*}
where $\kappa_p$ just depends on $p$. Now for any $\eps > 0$
$$ y | y|^{p-2} f(s,y,z,\psi) - \rho |y|^p \leq (\al - \rho) |y|^p + |y|^{p-1} g_s \leq (\al +\eps -  \rho) |y|^p + \frac{1}{p} \left( \frac{p\eps}{p-1}\right)^{1-p} g_s^p.$$
Therefore for any $\rho > \al$ we choose $\eps$ such that $\rho > \al+\eps$ and taking the expectation we have
\begin{eqnarray*}
&&\E e^{p\rho (t\wedge \tau)} |Y_{t\wedge \tau}|^p + \E \kappa_p \int_{t\wedge \tau}^{T \wedge \tau} e^{p\rho s}| Y_{s}|^{p-2} |  Z_s|^2 ds + \kappa_p  \E \int_{t\wedge \tau}^{T \wedge \tau} e^{p\rho s}| Y_{s^-}|^{p-2} d[ M ]_s \\
&& \qquad + \E \kappa_p\int_{t\wedge \tau}^{T \wedge \tau} e^{p\rho s} |Y_{s^-}  |^{p-2}  \| \psi_s\|_{L^2_\mu}^2 ds \\
&&\quad \leq \E e^{p\rho (T\wedge \tau)} |Y_{T\wedge \tau}|^p  + \left( \frac{p\eps}{p-1}\right)^{1-p} \E  \int_{t\wedge \tau}^{T \wedge \tau}  e^{p\rho s}  |g_s|^p ds.
\end{eqnarray*}
The same argument can be used in the case $1<p<2$.

\begin{Remark}
In dimension one, if $\xi$ and $f(t,0,0,0)$ are non negative, the $L^p$-solution $Y$ is non negative and if $f(s,0,z,\psi) \leq 0$ for any $z$ and $\psi$, the conclusion of Theorem \ref{thm:exis_Lp_sol_rnd_time} holds.
\end{Remark}

\begin{Remark}
In dimension one, under the assumptions of Theorem \ref{thm:exis_Lp_sol_rnd_time} (or of the previous remarks), and with condition (H3'), then the comparison result (Proposition \ref{prop:comp_sol_gene_BSDE}) holds.
\end{Remark}
Indeed we can sketch the proof to obtain that for any $0\leq t \leq T$
$$\what Y_{t\wedge \tau} \geq \E \left[ \Gamma_{t\wedge \tau,T\wedge \tau} \what Y_{T\wedge \tau} + \int_{t\wedge \tau}^{T\wedge \tau} \Gamma_{t\wedge \tau,s} f_s ds \bigg| \F_{t\wedge \tau} \right]$$
with suitable integrability conditions. The conclusion follows by letting $T$ go to $+\infty$.

\bibliography{biblio_revised_version}

\end{document}